\documentclass{amsart}
\usepackage{hyperref,enumerate}
\usepackage{xcolor}
\usepackage{comment}

\newcommand{\D}{\mathcal{D}}
\newcommand{\R}{\mathcal{R}}
\newcommand{\HH}{\mathcal{H}}
\newcommand{\ra}{\rangle}
\newcommand{\la}{\langle}

\newtheorem{theorem}{Theorem}[section]
\newtheorem{lemma}[theorem]{Lemma}
\newtheorem{proposition}[theorem]{Proposition}
\newtheorem{corollary}[theorem]{Corollary}
\newtheorem{claim}[theorem]{Claim}

\theoremstyle{definition}
\newtheorem{definition}[theorem]{Definition}

\theoremstyle{remark}
\newtheorem{remark}[theorem]{Remark}
\numberwithin{equation}{section}

\DeclareMathOperator*{\essinf}{ess\,inf}

\title[Extremal Metrics for conformal Dirichlet-to-Robin Map]{Existence and Regularity of Extremal Metrics for the Conformal Dirichlet-to-Robin Map}

\author{Samuel P\'{e}rez-Ayala}
\address{Department of Mathematics\\
         Princeton University\\
         Princeton, NJ 08544}



\begin{document}

\begin{abstract}
We study the variational properties of the spectrum of the Dirichlet-to-Robin map $\D_g$ on connected compact manifolds with boundary of dimension at least three. For the first eigenvalue, we show that Type II Yamabe metrics extremize the first normalized eigenvalue functional, and we characterize all extremals. If $[g]$ is a conformal class for which $\D_g$ has at least two negative eigenvalues, then we show the existence of a generalized metric that maximizes the second normalized eigenvalue of $\D_g$ in the conformal class. Moreover, we show that each such metric either defines a solution to an Escobar--Yamabe type equation on manifolds with boundary that changes sign along the boundary, or a weakly free-boundary harmonic map into the unit Euclidean ball. 
\end{abstract}

\thanks{This work was initiated during the author's first period at Princeton University. While there, the author was supported by NSF grant RTG-DMS-1502424.}

\maketitle

\tableofcontents

\section{Introduction}\label{Introduction}
Let $\overline{M}$ be a connected compact Riemannian manifold of dimension $n\ge 3$ with smooth boundary $\Sigma:=\partial \overline{M}$ and interior $M$, equipped with a conformal class $[g]$. Let $L_g$ denote the \textit{conformal Laplacian}, also known as the \textit{Yamabe operator}, which is defined by
\begin{equation}\label{DefCL}
L_g := -\Delta_g + c_nR_g.
\end{equation}
Here $\Delta_g:=\text{div}_g(\nabla_g )$ is defined as a negative operator, $c_n := \frac{n-2}{4(n-1)}$, and $R_g$ is the scalar curvature with respect to the metric $g$. This operator is elliptic and satisfies the following conformally covariant property: if $g_u= u^{\frac{4}{n-2}}g\in [g]$, then for all $\phi \in C^\infty(M)$ we have
\begin{equation}\label{ConfPropCL}
L_{g_u}(\phi) = u^{-\frac{n+2}{n-2}}L_g(\phi u).
\end{equation}

On a manifold with boundary there are many boundary eigenvalue problems associated with $L_g$. For instance, we could consider the following Dirichlet eigenvalue problem:
\begin{equation}\label{Dirichlet}
\begin{cases}
L_g(\phi) &= \lambda \phi \;(\text{in } M) \\
\phi&= 0 \; (\text{on }\Sigma).
\end{cases}
\end{equation}
We will refer to those numbers $\lambda\in\mathbb{R}$ for which there is a nonzero solution of (\ref{Dirichlet}) as \textit{Dirichlet eigenvalues}, and the collection of all such numbers will be denoted by $\text{Spec}(L_g^D)$. Notice that $\text{Spec}(L_g^D)$ is discrete, and that if $R_g=0$, then its first element $\lambda_1(L_g^D)$ is positive. 

We now proceed to introduce the eigenvalue boundary problem that will be of main interest to us. Denote by $\nu_g$ the outward unit normal vector along $\Sigma$, and by $h_g$ the (unnormalized) mean curvature of $\Sigma$. For every $\phi\in C^\infty(\overline{M})$, the \textit{Robin boundary operator} is defined as 
\begin{equation}\label{DefRobin}
B_g(\phi):= \partial_{\nu_g}\phi + 2c_nh_g\phi \hspace{.15in}(\text{on } \Sigma).
\end{equation}
$B_g$ is conformally covariant in the sense that, if $g_u=u^{\frac{4}{n-2}}g$, then 
\begin{equation}\label{ConfPropR}
B_{g_u}(\phi) = u^{-\frac{n}{n-2}}B_g(u\phi).
\end{equation}
Consider the numbers $\lambda\in \mathbb{R}$ for which there are nonzero solutions to the following system:
\begin{equation}\label{EigenR}
\begin{cases}
L_g(\phi) &= 0  \;(\text{in } M) \\
B_g(\phi) &= \lambda\phi \; (\text{on }\Sigma).
\end{cases}
\end{equation}
We refer to such numbers as \textit{Robin eigenvalues}. 

Assume that $0\not\in \text{Spec}(L_g^D)$, and let $\phi\in C^\infty(\Sigma)$ be given. Then there exists a unique solution $\mathcal{H}_g(\phi)\in C^\infty(\overline{M})$ (see \cite{Cox, Lassas}) of the system
\begin{equation}\label{Lg-HarmonicExtension}
\begin{cases}
L_g(\mathcal{H}_g(\phi)) &=0 \;(\text{in } M)\\
\mathcal{H}_g(\phi) & = \phi \;(\text{on }\Sigma).  
\end{cases}
\end{equation}
The \textit{conformal Dirichlet-to-Robin map} $\mathcal{D}_g: C^\infty(\Sigma)\to C^{\infty}(\Sigma)$ is defined by
\begin{equation}
\mathcal{D}_g(\phi):=B_g(\mathcal{H}_g(\phi)).
\end{equation} 
The operator $\mathcal{D}_g$ is an elliptic and self-adjoint pseudodifferential operator of order 1. Therefore, it has a discrete spectrum, which we can arrange in non-decreasing order as 
\begin{equation}\label{SpectrumDR}
\lambda_1(\mathcal{D}_g) \le \lambda_2(\mathcal{D}_g) \le \cdots \le \lambda_k(\mathcal{D}_g) \to\infty.
\end{equation}
As usual, each eigenvalue in (\ref{SpectrumDR}) is repeated according to its multiplicity. It is straightforward to verify that the spectrum of $\mathcal{D}_g$ and that of $B_g$, the collection of Robin eigenvalues, coincide. We would like to remark that if $0\in\text{Spec}(L_g^D)$, then there is a way to define the operator $\mathcal{D}_g$ by imposing an orthogonality condition on the domain functions $C^\infty(\Sigma)$ (see \cite{Cox}); this will not be used in this work. 

Due to the conformal properties (\ref{ConfPropCL}) and (\ref{ConfPropR}), there are many spectrally-defined quantities that give rise to conformal invariants which are of importance to us:
\begin{enumerate}[(i)]
\item The sign of $\lambda_1(\mathcal{D}_g)$ is an invariant of the conformal class $[g]$ (see Proposition 1.2 in \cite{Escobar2}). In fact, it agrees with the sign of the boundary analog of the Yamabe invariant, often called \textit{Type II Yamabe invariant} and defined by
\begin{equation}\label{YamInv}
Q(\overline{M},\Sigma,[g]):= \inf\; \frac{ \int_M(|\nabla_g\phi|^2+c_nR_g\phi^2)\;dv_g + 2c_n\int_\Sigma h_g\phi^2\;d\sigma_g}{\left(\int_\Sigma\phi^{\frac{2(n-1)}{(n-2)}}\;d\sigma_g\right)^{\frac{n-2}{n-1}}},
\end{equation}
where the infimum is taken over all $\phi\in C^\infty(\overline{M})$ for which $\phi\not=0$ on $\Sigma$, and where $d\sigma_g$ is the induced volume element on $\Sigma$. That $Q(\overline{M},\Sigma,[g])$ is conformally invariant was proved in \cite{Escobar}.

\item The dimension of $\text{Ker}(L_g^D)$ is conformally invariant. As a consequence, the condition $0\not\in \text{Spec}(L_g^D)$ is also conformally invariant. 

\item The number of negative eigenvalues of $\D_g$, which we denote by $N([g])$, is also a conformal invariant. In fact, it is equal to the number of negative eigenvalues of $L_g$ under the boundary condition $B_g(\phi) = 0$ on $\Sigma$, minus the number of negative eigenvalues of $L_g$ under Dirichlet conditions; see Theorem 6.1 in \cite{Cox}.
\end{enumerate}

We remark that the Yamabe invariant $Q(M^n,\Sigma, [g])$ can be $-\infty$, as it was  first pointed out by J. Zhiren (see \cite{Escobar3}). We assume that $Q(\overline{M},\Sigma,[g])$ is finite, that is, 
\begin{equation}\label{Assumption1}
Q(\overline{M},\Sigma,[g]) >-\infty.
\end{equation}
This is the case when the scalar curvature $R_g$ is nonnegative; see introduction in \cite{MarquesCAG}. Indeed, it is finite if and only if $\lambda_1(L_g^D)>0$ \cite{Escobar3}. On the other hand, it was proved in \cite{Escobar2} that
\begin{equation}
Q(\overline M, \Sigma, [g]) \le Q(\mathbb{B}^n, \mathbb{S}^{n-1}, [g_E]),
\end{equation}
where $\mathbb{B}^n$ denotes the closed unit ball in $\mathbb{R}^n$ endowed with the Euclidean metric $g_E$. 

\subsection{Statement of Results}

The main focus of our work will be finding extremals for the normalized eigenvalue functional $F_k$ ($k=1,2$) defined by
\begin{equation}\label{Functional}
F_k(u) := \lambda_k(u) \left(\int_\Sigma u^{\frac{2(n-1)}{n-2}}\;d\sigma_g\right)^{\frac{1}{n-1}},
\end{equation}
where $\lambda_k(u):=\lambda_k(\mathcal{D}_{g_u})$ and $g_u=u^{\frac{4}{n-2}}g\in [g]$. 

We start with the case of $k=1$.  In Section \ref{FirstEigenvalue}, we will show that if $-\infty<Q(\overline{M},\Sigma,[g]) < 0$, then 
\begin{equation}\label{FirstEigen1}
\sup_{g_u\in[g]}F_1(u) = Q(\overline{M},\Sigma,[g]).
\end{equation}
In the case where $Q(\overline{M},\Sigma,[g])\ge 0$, we have
\begin{equation}\label{FirstEigen2}
\inf_{g_u\in [g]} F_1(u) = Q(\overline{M},\Sigma,[g]).
\end{equation}
In particular, smooth extremal metrics for $F_1$ exist whenever smooth positive minimizers for $Q(\overline{M},\Sigma,[g])$ can be found. Geometrically, these minimizers provide a conformal metric which is scalar flat in the interior and has constant mean curvature along the boundary (\cite{Escobar2}). As discussed by Escobar in \cite{Escobar2}, such a minimizer always exists as long as 
\begin{equation}\label{ExistenceCond}
-\infty < Q(\overline{M},\Sigma,[g]) < Q(\mathbb{B}^n, \mathbb{S}^{n-1}, [g_E])
\end{equation}
holds. Inequality (\ref{ExistenceCond}) has been proved in multiple cases; see for instance \cite{Almaraz, Chen, Escobar2, MarquesCAG, Marques}. In summary, (\ref{ExistenceCond}) holds in the following cases: (1) $3\le n\le 7$; (2) $n\ge 8$ and $\overline{M}$ is spin; and (3) $n\ge 8$ and $\overline{M}$ is locally conformally flat \cite{ChenLaiWang}. The reader should notice that if the Yamabe invariant $Q(\overline{M},\Sigma, [g])$ is finite and negative, then (\ref{ExistenceCond}) trivially holds; thus a minimizer exists and it is scalar flat with constant mean curvature on the boundary.  

We summarize the previous discussion as our first theorem, which is proved in Section \ref{FirstEigenvalue}.

\begin{theorem}\label{TFirstEigenvalue}
Let $\overline{M}$ be a smooth Riemannian manifold of dimension $n\ge 3$ with smooth boundary $\Sigma:= \partial M$ and interior $M$, equipped with a conformal class $[g]$.
\begin{enumerate}
\item If $Q(\overline{M},\Sigma,[g])< 0$, then (\ref{FirstEigen1}) holds.  
\item If $Q(\overline{M},\Sigma,[g]) > 0$ and (\ref{ExistenceCond}) is satisfied, then (\ref{FirstEigen2}) holds. 
\item If $Q(\overline{M},\Sigma,[g])=0$, then $F_1\equiv 0$ on $[g]$.
\end{enumerate}
In either nonzero case, the extremal value is attained by a minimizer $g_Y\in [g]$ of $Q(\overline{M},\Sigma, [g])$; these metrics have zero scalar curvature in the interior and constant mean curvature at the boundary. In the negative case, after choosing as a background metric $g_Y$, a metric $g_u = u^{\frac{4}{n-2}}g_Y$ attains the maximal value for $F_1$ if and only if $u|_{\Sigma}$ is constant, that is, if $g_u|_{T\Sigma}$ is homothetic to ${g_Y}|_{T\Sigma}$. In the positive case, after fixing the background metric $g_Y$, a metric $g_u=u^{\frac{4}{n-2}}g_Y$ attains the minimal value of $F_1$ if and only if there exists a positive first generalized eigenfunction $\phi_1$ associated
to $g_u$ such that $u=C\phi_1$ on $\Sigma$ for some $C>0$, and
$\phi_1^{\frac{4}{n-2}}g_Y$ is a Type II Yamabe minimizer.
\end{theorem}

Theorem \ref{TFirstEigenvalue} provides the geometric motivation to study existence of extremals for higher eigenvalues. We would like to make some brief remarks about the closed case, i.e. the conformal Laplacian on compact manifolds $X^n$ without boundary. For the first eigenvalue, an analogous result to Theorem \ref{TFirstEigenvalue} holds, where $Q(\overline{M},\Sigma,[g])$ is replaced by the standard Yamabe invariant $Y(X,[g])$; see \cite{Ammann, Gursky} for a discussion. 

To motivate our second result, we make some further remarks on the closed case. We use $X$ to denote closed manifolds and reserve $\overline M$ for our nonclosed setup. In the case of positive Yamabe invariant, Ammann and Humbert in \cite{Ammann}, motivated by the equality
\begin{equation}
\inf_{g_u\in [g]}\lambda_1(L_{g_u})\text{Vol}(X,g_u)^{\frac{2}{n}} = Y(X,[g]),
\end{equation}
defined the second Yamabe invariant by setting 
\begin{equation}
\underline{\mu} _2(X,[g]):= \inf_{g_u\in [g]}\lambda_2(L_{g_u})\text{Vol}(X,dv_{g_u})^{\frac{2}{n}}.
\end{equation}
Unlike the case of the first eigenvalue, smooth minimizers do not exist. Under certain assumptions on $\underline\mu_2(X,[g])$, Ammann and Humbert proved the existence of a minimizer in the space of \textit{generalized conformal metrics}; see \cite{Ammann,Gursky} for definition in the closed case, or see Section \ref{Preliminaries} below for definition in our setup. Moreover, for any such minimizer $u$, there is a sign-changing solution $w\in C^{3,\alpha}(X)$ to a Yamabe-type equation such that $u=|w|$ holds (see Theorem 1.6 in \cite{Ammann}). A slight generalization of this result was obtained in \cite{ElSayed} for the case where the Yamabe invariant could be negative, but $\lambda_2(L_g)$ is still nonnegative.

We conclude our discussion of the closed case by explaining what is known in the case where $\lambda_2(L_g)<0$. In this case, it was proved in \cite{ElSayed} that the infimum of the normalized second eigenvalue functional over generalized conformal metrics is $-\infty$. Instead, and motivated by the equality
\begin{equation}
\sup_{g_u\in [g]}\lambda_1(L_{g_u})\text{Vol}(X,g_u)^{\frac{2}{n}} = Y(X,[g]) \hspace{.3in}   (\text{if }Y(X,[g])<0),
\end{equation}
the author and M.J. Gursky considered
\begin{equation}
\overline{\mu} _2(X,[g]):= \sup_{g_u\in [g]}\lambda_2(L_{g_u})\text{Vol}(X,dv_{g_u})^{\frac{2}{n}}.
\end{equation}
We proved the existence of maximizers in the space of generalized conformal metrics; see Theorem 1.1 in \cite{Gursky}. Moreover, associated with each such maximizer, there is either an associated nodal solution to a Yamabe-type equation, or a harmonic map into a sphere; see Corollary 1.2 in \cite{Gursky}. We also provide examples where both cases occur and where a maximizer is smooth. 

Our next result is a generalization of this to the case of compact manifolds with boundary. 

\begin{theorem}\label{MainTheorem}
Let $(\overline{M},g)$ be a smooth connected Riemannian manifold of dimension $n\ge 3$ with smooth boundary $\Sigma= \partial M$ and interior $M$, equipped with a conformal class $[g]$ for which $Q(\overline{M},\Sigma, [g])$ is finite. Assume further that $N([g])\ge 2$ and $0\not \in \text{Spec}(\D_g)$. Choose a Type II
Yamabe representative \(g\in[g]\), normalized so that
\[
R_g=0,\qquad h_g<0\ \text{is constant},\qquad
\operatorname{Vol}(\Sigma, d\sigma_g)=1.
\]
Then there is a nonnegative function $ u\in C^{0,\beta}(\Sigma)\cap C^\infty(\Sigma\setminus u^{-1}(0))$ with $\beta\in (0,1)$, whose $g$-harmonic extension $\hat u$ is smooth and positive in the interior $M$, that maximizes the normalized second eigenvalue functional 
\begin{equation}\label{MT1}
F_2: u\in L_{>0}^{\frac{2(n-1)}{n-2}}(\Sigma)\longmapsto \lambda_2(u)\left(\int_\Sigma u^{\frac{2(n-1)}{n-2}}\;d\sigma_g\right)^{\frac{1}{n-1}}
\end{equation} 

Furthermore, $u^{-1}(0)$ has zero boundary measure -- specifically, $\dim_{\mathrm{H}}(u^{-1}(0))\le n-2$. Moreover, for this maximizer $u$, there exists a finite collection $\{\phi_i\}_{i=1}^k\subset C^{1,\beta}(\Sigma)\cap C^\infty(M)$ of second generalized eigenfunctions satisfying
\begin{equation}\label{MT2}
\sum_{i=1}^k\phi_i^2 = u^2
\end{equation}
on $\Sigma$. Here $1\le k\le \text{dim}(E_2(u))$, where $E_2(u)$ is the space of second generalized eigenfunctions associated to $ u$.
\end{theorem}

Examples of smooth connected manifolds with smooth boundary for which our assumptions are satisfied are easy to construct. For instance, we can take a connected closed flat manifold, like a flat torus $\mathbb{T}^n$, remove a small geodesic ball and equip the resulting manifold with the restricted flat metric and its boundary with the outward unit normal. The resulting manifold has finite and negative $Q(\overline M, \Sigma, [g])$, and so $N([g])\ge 1$. Removing finitely many pairwise disjoint geodesic balls gives examples with
several boundary components, and if the balls have the same radius, then the boundary mean curvature has the same negative constant value on every component. To get higher $N([g])$, we can take a flat manifold $X^{n-1}$ with constant and negative mean curvature, as described above, and consider the product manifold $M^n = X^{n-1}\times S_l$ equipped with the product metric $g_{prod}$, where $S_l$ denotes the circle of radius $l>0$. By testing in the Rayleigh quotient the span of $1$, $\cos(s/l)$ and $\sin(s/l)$, the latter two defined on the circle factor, we obtain $N([g_{prod}])\ge 3$ after taking $l$ sufficiently large. 

We remark that (\ref{MT2}), and consequently Corollary \ref{NS-HM} below, holds for any maximizer of $F_2$; see remark after Proposition \ref{EulerEq}. An important consequence of Theorem \ref{MainTheorem} is Corollary \ref{NS-HM} below. Denote by $g_E$ the standard Euclidean metric on $\mathbb{R}^k$. We say that a map $\Phi: (\overline{M},g) \to (\mathbb{B}^k,g_E)$ with $\Phi(\Sigma) \subset \mathbb{S}^{k-1}$ is a free-boundary harmonic map if $\Delta_g \Phi = 0$ in $M$ and $\partial_\nu \Phi$ is parallel to $\Phi$ along $\Sigma$. We say that a map $\Phi:(\overline{M}, g)\to \mathbb{B}^k$ is weakly free-boundary harmonic if $\Phi(x)\in \mathbb{S}^{k-1}$ for a.e. $x\in \Sigma$ and if for any $V\in L^\infty\cap W^{1,2}(\overline M, \mathbb{R}^{k})$ with $V(x)\in T_{\Phi(x)}\mathbb{S}^{k-1}$ for a.e. $x\in \Sigma$, $\int_Mg(\nabla_g\Phi, \nabla_g V)\;dv_g = 0$; see \cite{Jost, LaurainPetrides} for a more general and formal definition. This definition can be generalized for metrics $g_u = u^{\frac{4}{n-2}}g$ which possibly degenerate on a (boundary) zero measure set on $\Sigma$, with the additional feature that now the mentioned Sobolev spaces will depend on the metric $g_{u}$; see discussion in Section \ref{Tlimit}.

\begin{corollary}\label{NS-HM}
Let $u\in C^{0,\beta}(\Sigma)\cap C^\infty(\Sigma\setminus u^{-1}(0))$ be a maximizer provided by Theorem \ref{MainTheorem}. Then
\begin{enumerate}
\item If $k=1$, then $u = |\phi|$ along $\Sigma$, where $\phi\in C^{1,\beta}(\Sigma)\cap C^\infty(M)$ is a sign-changing solution on $\Sigma$ of
\begin{equation}\label{NodalCase}
\begin{cases}
\Delta_g \phi &= 0 \text{ in } M,\\
B_g(\phi) &= \lambda_2(u)\phi |\phi|^{\frac{2}{n-2}} \text{ on }\Sigma.
\end{cases}
\end{equation}
\item Assume $k>1$ and set $\Sigma^*:=\Sigma\setminus u^{-1}(0)$, and recall that $\dim_{\mathcal{H}}(u^{-1}(0))\le n-2$. Then the map $\Phi=(\frac{\phi_1}{\hat u},\cdots, \frac{\phi_k}{\hat u}): (M^n, g_{\hat u})\to (\mathbb{B}^k,g_E)$ satisfies $\Phi(\Sigma^*)\subset \mathbb{S}^{k-1}$ and it is a harmonic map with $\partial_{\nu_{g_{\hat u}}}\Phi \| \Phi$ on $\Sigma^*$. Here $g_{\hat u} = \hat u^{\frac{4}{n-2}}g$ is the conformal metric defined by the harmonic extension $\hat u$ of $u$, positive in $M$, which agrees with $u^{\frac{4}{n-2}}g$ along $\Sigma$. Therefore, $\Phi$ is a $g_{\hat u}$-weakly free-boundary harmonic map.
\end{enumerate}
\end{corollary} 

An example where situation (2) occurs is not trivial and will be the focus of the author's upcoming work. If $N([g]) = 2$, then $(1)$ occurs and we have a sign-changing solution to \ref{NodalCase}. Positive solutions of this equation have been widely studied, but the problem of finding sign-changing solutions remains largely open. In \cite{ClappPellacciPistoia}, the existence of sign-changing solutions for equations that include (\ref{NodalCase}) is proved on manifolds for which $Q(\overline M,\Sigma,[g])>0$. Our manifolds, on the other hand, satisfy $-\infty< Q(\overline M, \Sigma, [g])<0$. 

Although Robin eigenvalue optimization has been studied, the standard Robin problem is not conformally invariant. To our knowledge, eigenvalue optimization for the eigenvalues of the conformal Dirichlet-to-Robin map is new. It would be interesting to study the cases where $N([g])=1$ or $N([g]) = 0$, and try to show existence of minimizers in the space of generalized conformal metrics as it was done in \cite{Ammann,ElSayed} for closed manifolds. However, in such cases, based on variational properties of the eigenvalue functional, we do not expect any connection with existence of free-boundary harmonic maps into Euclidean balls; see Remark 6.1 in \cite{Gursky}. 

In the case of surfaces with boundary, a substantial theory has been
developed for the maximization of Steklov eigenvalues; see for instance \cite{Fraser2,Petrides2,VinokurovSymmetries} and references therein. Under suitable assumptions, maximizing metrics are associated with free-boundary minimal immersions into Euclidean balls; see Theorem 1 in \cite{Petrides2}. The connection with free-boundary harmonic maps into Euclidean balls in this case arises if we maximize the $k$-th Steklov eigenvalue in the space of conformal metrics with fixed boundary length; see Proposition 2.8 in \cite{Fraser}. For some existence and regularity results for maximization of Steklov eigenvalues in higher dimensions, see the work of Vinokurov in \cite{VinokurovHigher}. There, however, the free-boundary harmonic maps obtained are into the infinite-dimensional ball $\mathbb{B}^\infty$, and it is only known to be finite dimensional when the domain manifold $\overline{M}$ is low-dimensional. In our case, due to a universal bound on the multiplicity, the target Euclidean ball is always finite dimensional. 

Related variational problems for other conformally covariant operators have been studied; some works include \cite{GonzalezSaez, HumbertPetridesPremoselli, PerezAyalaSireXu, PerezAyalaPaneitz}.

\subsection{Plan of the paper} In the next section, we present some definitions, discuss important conformal properties (some introduced in \cite{Cox}), and give a detailed treatment of generalized eigenvalues and eigenfunctions. We remark that Proposition \ref{Estimates} could be of independent interest; there we derive interior and boundary $L^\infty$ estimates in a scenario where the potential function lies in $L^{n-1}(\Sigma)$ but has a fixed sign. The proof of Theorem \ref{TFirstEigenvalue} is discussed in Section \ref{FirstEigenvalue}. In Section \ref{FirstVariationFormulas} we derive first variation type formulas for the normalized eigenvalue functional. The difficulty here comes from the non-differentiability of $g_u\longmapsto \lambda_2(\D_{g_u})$, although its derivation is now standard in the literature. 

The proof of the main theorem, Theorem \ref{MainTheorem}, is contained in Sections \ref{eRegularization}, \ref{UniformEstimates}, and \ref{Tlimit}. The techniques for its proof are motivated by those introduced by the author and M.J. Gursky in \cite{Gursky}. However, there are several new key features and complications. Here the generalized conformal factor is defined on the boundary, and so compactness is governed by the critical Sobolev trace embedding instead. Moreover, the conformal Dirichlet-to-Robin map is an elliptic pseudodifferential operator of order one on the boundary, so both the regularity theory and the analysis of boundary eigenfunctions require new arguments in many parts; see for instance Proposition \ref{ExistenceSimplicity} about boundary positivity (a.e.) of the first generalized eigenfunction or Proposition \ref{u-HolderBound} for a uniform boundary estimate for the sequence of extremals. A particularly delicate point is the possibility that the limiting conformal
factor vanishes along the boundary. Unique continuation principles from sets of positive measure on the boundary are quite subtle; we rely on a recent result by Z. Li in \cite{Li} to discuss the size of the zero set of our extremal function $u$ in Section \ref{Tlimit}. From there we obtain regularity and show the maximality of our limiting function.
\section{Preliminaries}\label{Preliminaries}

\noindent {\bf Assumptions and conventions.} We assume $\overline{M}$ is a compact connected Riemannian manifold of dimension $n\ge 3$ with smooth boundary $\Sigma:= \partial M$ and interior $M$, equipped with a conformal class $[g]$. With the exception of Section \ref{FirstEigenvalue}, we always assume that the Type II Yamabe invariant $Q(\overline{M},\Sigma,[g])$ is finite and negative. Therefore, and as a consequence of Proposition 1.4 in \cite{Escobar2}, up to a conformal change of metric, we can assume that the background metric $g$ satisfies $R_g=0$,  $h_g$ is a negative constant, and $\text{Vol}(\Sigma,d\sigma_g)=1$. As noticed in \cite{Escobar3}, the finiteness of $Q(\overline{M},\Sigma, [g])$ is equivalent to $\lambda_1(L_g^D)$ being positive; consequently, $0\not\in\text{Spec}(L_g^D)$. In particular, the $L_g$-harmonic extension discussed in (\ref{Lg-HarmonicExtension}) is unique. Finally, we assume that $N([g])\ge 2$ and $0\not\in \text{Spec}(\D_g)$.

We usually omit writing the Riemannian volume measure $dv_g$ or the Riemannian boundary measure $d\sigma_g$ when dealing with Sobolev spaces. That is, we write $W^{1,2}(M)$, $L^q(M)$, and $L^q(\Sigma)$ instead of $W^{1,2}(M,dv_g)$, $L^q(M,dv_g)$, and $L^q(\Sigma, d\sigma_g)$, respectively. 

\subsection{$\mathcal{H}_g$ under conformal transformation} 

Let $u\in C^{\infty}(\overline{M})$ be a strictly positive function and consider a conformal metric $g_u=u^{\frac{4}{n-2}}g \in [g]$. The outward unit normal derivative and the mean curvature on $\Sigma$ transform according to
 \begin{equation}
 \partial_{\nu_{g_u}} = u^{-\frac{2}{n-2}}\partial_{\nu_g} \hspace{.2in}\text{ and }\hspace{.2in}h_{g_u} = u^{-\frac{2}{n-2}}\left(h_g + \frac{2(n-1)}{n-2}u^{-1}\partial_{\nu_g}u\right).
 \end{equation}
 In particular, $\nu_{g_u}$ does not change direction. These imply that $B_g$ is conformally covariant in the following sense: for any $\phi\in C^\infty(\overline{M})$, 
 \begin{equation}\label{RCL}
 B_{g_u}(\phi) = u^{-\frac{n}{n-2}}B_g(u\phi).
 \end{equation}
 
 We proceed to explain the conformal properties of the operator $\HH_g$. As a consequence, we also derive the conformal transformation law for $\D_g$. Suppose we are given a function $\phi\in C^\infty(\Sigma)$. Since $0\not\in\text{Spec}(L_g^D)$, there exists a unique function $\HH_{g_u}(\phi)\in C^\infty(\overline{M})$ such that $L_{g_u}(\HH_{g_u}(\phi)) = 0$ (\cite{Cox,Lassas}). Then we take $\psi := u_{|_{\Sigma}}\phi \in C^\infty(\Sigma)$, and notice that $u\HH_{g_u}(\phi)$ is the correct extension with respect to $L_g$. Indeed, $(u\HH_{g_u}(\phi))_{|_{\Sigma}} = u_{|_{\Sigma}}\phi = \psi$ on $\Sigma$ and, due to the conformal transformation (\ref{ConfPropCL}),
\begin{equation}
L_g(u\HH_{g_u}(\phi)) = u^{\frac{n+2}{n-2}}L_{g_u}(\HH_{g_u}(\phi)) = 0.
\end{equation} 
By uniqueness, we therefore have
\begin{equation}
\HH_g(\psi) = \HH_g(u_{|_{\Sigma}}\phi)= u\HH_{g_u}(\phi),
\end{equation}
thus 
\begin{equation}
\begin{split}
\D_{g_u}(\phi) &= B_{g_u}(\HH_{g_u}(\phi)) = u^{-\frac{n}{n-2}}B_g(u\HH_{g_u}(\phi)) = u^{-\frac{n}{n-2}} B_g(\HH_g(\psi))\\ &= u^{-\frac{n}{n-2}} \D_g(\psi) = u^{-\frac{n}{n-2}}\D_g((u_{|_{\Sigma}})\phi).
\end{split}
\end{equation}
In summary,
\begin{equation}\label{DCL}
\D_{g_u}(\phi) = (u_{|_{\Sigma}})^{-\frac{n}{n-2}}\D_g((u_{|_{\Sigma}})\phi).
\end{equation}

These properties are discussed in \cite{Cox} and are only included here for the convenience of the reader.

\subsection{Generalized Eigenvalues and Eigenfunctions}

Previous work on maximizing eigenvalues on (closed) manifolds indicates that extremal metrics are expected to possess singularities; see \cite{Ammann, Gursky, KarpukhinNadirashviliPenskoiPolterovich, KarpukhinStern,Petrides, Petrides2} and references therein. Therefore, the class $[g]$ of smooth conformal Riemannian metrics is not appropriate for our purposes, and a different variational setting is needed (see \cite{Kokarev} for a further discussion on this topic).  We take the approach developed by Ammann-Humbert in \cite{Ammann} (see also \cite{ElSayed}) and by M. J. Gursky and the author in \cite{Gursky} via \textit{generalized eigenvalues}. The purpose of this section is to explain this approach. 

The variational characterization for $\lambda_k(\D_g)$ is given by
\begin{equation}\label{VarCar1}
\lambda_k(\D_g)= \inf_{S_k\subset W^{1,2}(M)} \sup_{\phi\in S_k,\;\phi|_\Sigma \not \equiv 0} \R_g(\phi),
\end{equation}
where $S_k$ is a $k$-dimensional subspace of functions on which the trace map is injective, and
\begin{equation}
\R_g(\phi) := \frac{\int_M |\nabla_g\phi|^2 \;dv_g+2c_n\int_\Sigma h_g\phi^2 \;d\sigma_g}{\int_\Sigma \phi^2 \;d\sigma_g}
\end{equation}
is the associated Rayleigh quotient. If we consider a conformal metric $g_u = u^{\frac{4}{n-2}}g$, then
\[
\begin{split}
\R_{g_u}(\phi) & = \frac{\int_M (|\nabla_{g_u}\phi|^2 + c_nR_{g_u}\phi^2)\;dv_{g_u}+2c_n\int_\Sigma h_{g_u}\phi^2 \;d\sigma_{g_u}}{\int_\Sigma \phi^2 \;d\sigma_{g_u}} \\ & = \frac{\int_M \phi L_{g_u}(\phi) \;dv_{g_u} + \int_\Sigma \phi\partial_{\nu_{g_u}}\phi\;d\sigma_{g_u}  + 2c_n\int_\Sigma h_g\phi^2u^2 \;d\sigma_g + \int_\Sigma \phi^2u\partial_{\nu_g}u\;d\sigma_g}{\int_\Sigma \phi^2 u^{\frac{2(n-1)}{n-2}}\;d\sigma_g} \\ &= \frac{\int_M (u\phi)L_g(u\phi)\;dv_g  + 2c_n\int_\Sigma h_g(u\phi)^2\;d\sigma_g + \int_\Sigma (u\phi)\langle\nabla_g(u\phi),\nu_g\rangle\;d\sigma_g}{\int_\Sigma (u\phi)^2u^{\frac{2}{n-2}}\;d\sigma_g} \\ &= \frac{\int_M |\nabla_g(u\phi)|^2 \;dv_g + 2c_n\int_\Sigma h_g(u\phi)^2\;d\sigma_g}{\int_\Sigma (u\phi)^2u^{\frac{2}{n-2}}\;d\sigma_g} \\ &= \frac{\int_M|\nabla_g\psi|^2 \;dv_g + 2c_n\int_\Sigma h_g\psi^2\;d\sigma_g}{\int_\Sigma \psi^2u^{\frac{2}{n-2}}\;d\sigma_g} := \R_g^u(\psi),
\end{split}
\]
where $\psi=u\phi$. One of the advantages of working with $\R_g^u$ is that the conformal factor $u$ appears only in the denominator, and so $\R_g^u$ only depends on $u$ through its boundary values. Using $\R_g^u$, we can write the variational characterization for $\lambda_k(u):= \lambda_k(\D_{g_u})$ as follows:
\begin{equation}\label{VarCar2}
\lambda_k(u)= \inf_{S_k\subset W^{1,2}(M)} \sup_{\phi\in S_k,\;\phi|_\Sigma \not \equiv 0} \R_g^u(\phi).
\end{equation}
In this convention, $\lambda_k(1)$ is the $k$-th eigenvalue with respect to the background metric $g$. Studying extremal functions for $\R_g^u$ leads to the new system
\begin{equation}\label{EigenR2}
\begin{cases}
L_g(\phi) &= 0  \;(\text{in }M) \\
B_g(\phi) &= \lambda\phi u^{\frac{2}{n-2}} \;(\text{on }\Sigma),
\end{cases}
\end{equation}
which can also be viewed as (\ref{EigenR}) under conformal change. This motivates the following discussion. 

Let $N=\frac{2(n-1)}{n-2}$, and define
\begin{equation}
L^N_{\ge 0}(\Sigma):= \{u\in L^N(\Sigma): u\ge 0\}\setminus \{0\}.
\end{equation}
We will refer to elements in $L^N_{\ge 0}(\Sigma)$ as \textit{generalized conformal factors}. Similarly, we define
\begin{equation}
L^N_{>0}(\Sigma):= \{u\in L^N(\Sigma): u^{-1}(0)\text{ has zero (boundary) measure}\}.
\end{equation}
The number $N$ is the critical exponent in the Sobolev trace embedding
\begin{equation}\label{STE}
W^{1,2}(M)\hookrightarrow L^q(\Sigma),
\end{equation}
meaning that (\ref{STE}) is compact for $1< q < N$ and only continuous for $q=N$. We repeatedly use Friedrichs's inequality (see Theorem 4.1.7 in \cite{Hsiao}), which states that
\begin{equation}\label{FriedrichsInequality}
\int_M \phi^2 \;dv_g \le C \left(\int_M |\nabla_g \phi|^2\;dv_g + \int_\Sigma \phi^2\;d\sigma_g\right)
\end{equation}
for any $\phi\in W^{1,2}(M)$. Also, we say that a $k$-dimensional subspace $S_k=\text{span}\{\phi_1,\cdots,\phi_k\}$ of $W^{1,2}(M)$ belongs to the $k$-th modified Grassmannian on the boundary, denoted by $Gr_k^u(W^{1,2}(M)) $, if and only if the functions $\phi_1u^{\frac{N-2}{2}},\cdots, \phi_k u^{\frac{N-2}{2}}$ are linearly independent on $\Sigma$. This linear independence is understood in $L^2(\Sigma)$.

We can now define what generalized eigenvalues are in this context.

\begin{definition}\label{GeneralizedEigen}
For $u\in L^N_{\ge 0}(\Sigma)$, we define the \textit{k-th generalized eigenvalue} $\lambda_k(u)$ by 
\begin{equation}\label{GeneralizedEigen2}
\lambda_k(u):= \inf_{S_k\in Gr^u_k(W^{1,2}(M))} \sup_{\phi\in S_k,\;\phi|_\Sigma \not \equiv 0} \R_g^u(\phi).
\end{equation}
If $\lambda_k(u)$ is finite, then a \textit{k-th generalized eigenfunction} $\phi_k$ means a weak $W^{1,2}(M)$- solution of the following system:
\begin{equation}\label{EigenValEq}
\begin{cases}
L_g(\phi) &= 0  \;(\text{in } M) \\
B_g(\phi) &= \lambda_k(u)\phi u^{N-2} \;(\text{on }\Sigma).
\end{cases}
\end{equation}
The collection of $k$-th generalized eigenfunctions associated to $u\in L^N_{\ge 0}(\Sigma)$ is called the \textit{k-th generalized eigenspace} and it is denoted by $E_k(u)$.
\end{definition}

\begin{remark}
Any generalized $k$-th eigenfunction $\phi$ cannot vanish on $\Sigma$, i.e. $\phi_{|_\Sigma}\not \equiv 0$. Otherwise, we would have $0\in\text{Spec}(L_g^D)$, contradicting our assumption. Also, thanks to (\ref{STE}), any function $\phi\in W^{1,2}(M)$ makes sense when restricted to the boundary, i.e. its trace is well defined. We abuse notation and write $\phi_{|_\Sigma}$ or ``$\phi$ on $\Sigma$'' instead of $T(\phi)$, where $T: W^{1,2}(M) \to L^q(\Sigma)$ is the trace operator.
\end{remark}

We start by discussing the number of negative generalized eigenvalues. In the smooth setting, the number $N([g])$ of negative eigenvalues for the Conformal Dirichlet-to-Robin map is a conformal invariant; see Section \ref{Introduction}. We prove that this is still the case when working with generalized conformal factors if we require $u^{-1}(0)$ to have zero (boundary) measure.  We will need the following observation made in \cite{Gursky}; see Lemma 2.2:

\begin{lemma}\label{TestFunctions}
Let $u,w\in L^N_{\ge 0}(\Sigma)$ be generalized conformal factors. If $\{x\in \Sigma: u(x)>0\}\subseteq \{x\in\Sigma: w(x)>0\}$, then $Gr_k^u(W^{1,2}(M)) \subseteq Gr_k^w(W^{1,2}(M))$.
\end{lemma}

We are now in a position to prove that the number of negative eigenvalues remains unchanged as long as our generalized conformal factors are in $L^N_{>0}(\Sigma)$; see Proposition 2.5 in \cite{Gursky}. We include the argument here for completeness.

\begin{proposition}\label{NegativeEigen}
If $u\in L^N_{> 0}(\Sigma)$, then the number of negative generalized eigenvalues associated to $u$ is $N([g])$.
\end{proposition}
\begin{proof}
We want to show that $N(u) = N([g])$, where $N(u)$ denotes the number of negative generalized eigenvalues associated to $u$. Given any $k\in\mathbb{N}$, it is enough to prove that $\lambda_k(1)<0$ if and only if $\lambda_k(u)<0$. If $\lambda_k(1)\ge 0$, then 
\begin{equation}
\sup_{\phi\in S_k, \;\phi|_\Sigma \not \equiv 0} \R_g^1(\phi)\ge 0,
\end{equation} 
for all $S_k\in Gr_k^1(W^{1,2})$. By Lemma \ref{TestFunctions}, $Gr^u_k(W^{1,2}(M))\subseteq Gr_k^1(W^{1,2}(M))$, and thus $\lambda_k(u)\ge 0$. This shows that $N([g])\ge N(u)$. To prove the opposite inequality, just notice that, because $u^{-1}(0)$ has zero measure, we have $Gr_k^u(W^{1,2}(M)) = Gr_k^1(W^{1,2}(M))$. A similar argument then shows that $N(u)\ge N([g])$.
\end{proof}

Our next goal is to discuss the existence of generalized eigenfunctions. In order to do so, it is necessary to first discuss the finiteness of $u\mapsto \lambda_1(u)$. This is nontrivial as there are examples of smooth generalized conformal factors $u\in L^N_{\ge 0}(\Sigma)$ for which $\lambda_1(u) = -\infty$; this follows from an argument similar to the one used in Proposition 2.1 of \cite{ElSayed}, with the adjustment that the constructed $\epsilon$-ball needs to be centered at a boundary point.

\begin{proposition}\label{finite}
If $u\in L^N_{> 0}(\Sigma)$, then $\lambda_1(u)>-\infty$. Moreover, any sequence $\{\psi_j\}_{j=1}^\infty\subset W^{1,2}(M)$, normalized so that $\int_\Sigma \psi_j^2u^{N-2}\;d\sigma_g = 1$, satisfying 
\begin{equation}\label{finite0}
\int_M |\nabla_g\psi_j|^2\;dv_g + 2c_nh_g\int_\Sigma \psi_j^2\;d\sigma_g =: \lambda_j < 0,
\end{equation}
is bounded in $W^{1,2}(M)$.
\end{proposition}
Notice that Proposition \ref{finite} implies that normalized minimizing sequences for $\lambda_1(u)$, or for higher order negative eigenvalues, are bounded in $W^{1,2}(M)$.
\begin{proof}
Recall that, by assumption, $N([g])\ge 2$, and so $\lambda_1(u)\le \lambda_2(u)<0$ by Proposition \ref{NegativeEigen}. Let $\{f_j\}_{j=1}^\infty \subset W^{1,2}(M)$ be a minimizing sequence for $\lambda_1(u)$:
\begin{equation}\label{finite1}
\frac{\int_M|\nabla_gf_j|^2 \;dv_g + 2c_n\int_\Sigma h_gf_j^2\;d\sigma_g}{\int_\Sigma f_j^2u^{N-2}\;d\sigma_g} := \lambda_j\le 0.
\end{equation}
By the homogeneity properties of the Rayleigh quotient, we can assume that the $f_j$'s are normalized such that 
\begin{equation}
\int_\Sigma f_j^2u^{N-2}\;d\sigma_g = 1.
\end{equation}
We claim that the sequence $\{f_j\}_{j=1}^\infty$ is bounded in $W^{1,2}(M)$. Note that this would be enough to prove boundedness of $\lambda_1(u)$ thanks to (\ref{STE}). Our argument is general and will also show the second part of the statement. 

Let us suppose on the contrary that $\|f_j\|_{W^{1,2}(M)} \to \infty$ as $j\to\infty$, and define $\tilde f_j := f_j\|f_j\|_{W^{1,2}(M)}^{-1}$.
The new sequence $\{\tilde f_j\}_{j=1}^\infty$ is bounded in $W^{1,2}(M)$ with constant norm equal to 1. By standard compactness arguments, we deduce the existence of a function $\tilde f \in W^{1,2}(M)$ such that 
\begin{equation}\label{finite2}
\begin{cases}
\tilde f_j &\rightharpoonup \tilde f \text{ in } W^{1,2}(M) \\
\tilde f_j &\rightarrow \tilde f \text{ in } L^2(M).
\end{cases}
\end{equation}
Moreover, by the Sobolev trace embedding theorem (\ref{STE}), we have
\begin{equation}\label{finite3}
\tilde f_j \rightarrow \tilde f \text{ in } L^2(\Sigma).
\end{equation}

First, it follows from Fatou's lemma that
\begin{equation}
0\le \int_\Sigma\tilde f^2 u^{N-2}\;d\sigma_g \le \liminf_{j\to\infty} \int_\Sigma\tilde f_j^2u^{N-2}\;d\sigma_g = 0,
\end{equation}
and therefore, since $u>0$ a.e. on $\Sigma$, 
\begin{equation}\label{finite4}
\tilde f = 0 \text{ on } \Sigma.
\end{equation}
On the other hand, from (\ref{finite1}) we deduce
\begin{equation}\label{finite5}
\int_M|\nabla_g\tilde f_j|^2\;dv_g + 2c_n\int_\Sigma h_g\tilde f_j^2\;d\sigma_g = \lambda_j\int_\Sigma \tilde f_j^2u^{N-2}\;d\sigma_g\le0,
\end{equation}
and thus
\begin{equation}
\int_M|\nabla_g\tilde f_j|^2\;dv_g \le 2c_n (-h_g)\int_\Sigma \tilde f_j^2\;d\sigma_g.
\end{equation}
Using (\ref{finite2}), (\ref{finite3}), and the weak lower semicontinuity of the $W^{1,2}(M)$-norm, we obtain
\[
0\le \int_M|\nabla_g \tilde f|^2\;dv_g \le \liminf_{j\to\infty} \int_M |\nabla_g\tilde f_j|^2\;dv_g \le 2c_n\|h_g\|_\infty \lim_{j\to\infty} \int_\Sigma \tilde f_j^2\;d\sigma_g= 0
\]
This implies that $\tilde f$ is constant in $M$, and so $\tilde f \equiv 0$ on $\overline{M}$ by (\ref{finite4}). However, this is not possible:
\[
\begin{split}
1 & = \liminf_{j\to\infty} \left(\int_M|\nabla_g\tilde f_j|^2\;dv_g + \int_M\tilde  f_j^2\;dv_g\right) =\liminf_{j\to\infty}\int_M|\nabla_g\tilde f_j|^2\;dv_g + \int_M\tilde f^2\;dv_g \\ &= \int_M \tilde f^2\;dv_g.
\end{split}
\]
Therefore, the original sequence is bounded. This concludes the proof. Notice that for normalized sequences such as in (\ref{finite0}), we still have the correct sign as in (\ref{finite1}), and so the argument applies without requiring the boundedness of $\lambda_j$.
\end{proof}

\begin{proposition}\label{Boundedness-SeqConfFact}
Let $u_j\in L^N_{>0}(\Sigma)$ be a sequence of generalized conformal factors with $\int_\Sigma u_j^N\;d\sigma_g = 1$. If $\{\psi_j\}_{j=1}^\infty\subset W^{1,2}(M)$ with $\int_\Sigma \psi_j^2u_j^{N-2}\;d\sigma_g = 1$ is a sequence of weak solutions of  
\begin{equation}\label{Boundedness-SeqConfFact1}
\begin{cases}
\Delta_g(\psi_j) &= 0  \;(\text{in } M) \\
B_g(\psi_j) &= \lambda_j\psi_j u_j^{N-2} \;(\text{on }\Sigma).
\end{cases}
\end{equation}
and $0>\lambda_j\to \lambda$, then $\{\psi_j\}_{j=1}^\infty$ is bounded in $W^{1,2}(M)$.
\end{proposition}

\begin{proof}
Since 
\begin{equation}
\int_M|\nabla_g \psi_j|^2\;dv_g + 2c_nh_g\int_\Sigma \psi_j^2\;d\sigma_g = \lambda_j <0,
\end{equation}
we have $\int_M|\nabla_g\psi_j|^2\;dv_g < 2c_n(-h_g)\int_{\Sigma}\psi_j^2\;d\sigma_g$, and so, by Friedrichs inequality \ref{FriedrichsInequality}, it is enough to show boundedness in $L^2(\Sigma)$. To this end, let us assume the contrary, that is, suppose $\|\psi_j\|_{L^2(\Sigma)}\to \infty$ and set $\tilde \psi_j = \psi_j \|\psi_j\|_{L^2(\Sigma)}^{-1}$. From rescaling (\ref{Boundedness-SeqConfFact1}) we now obtain $\int_M|\nabla_g\tilde \psi_j|^2\;dv_g < 2c_n(-h_g)$, and so $\{\tilde \psi_j\}_{j=1}^\infty$ is bounded in $W^{1,2}(M)$ thanks to (\ref{FriedrichsInequality}). Denote by $\tilde \psi$ its limiting function obtained from standard compactness arguments. 

Using that
\begin{equation}\label{Boundedness-SeqConfFact2}
\int_M \langle\nabla_g\tilde \psi_j,\nabla_g\psi\rangle\;dv_g + 2c_nh_g\int_\Sigma \tilde \psi_j \psi\;d\sigma_g = \lambda_j \int_\Sigma \tilde \psi_j\psi u_{j}^{N-2}\;d\sigma_g
\end{equation}
holds for every $\psi\in C^\infty(\overline{M})$, weak convergence in $W^{1,2}(M)$ then implies that the left hand side of (\ref{Boundedness-SeqConfFact2}) converges to
\begin{equation}\label{Boundedness-SeqConfFact13}
\int_M\langle\nabla_g\tilde \psi,\nabla_g\psi\rangle \;dv_g + 2c_n h_g\int_\Sigma \tilde \psi\psi\;d\sigma_g. 
\end{equation}
On the other hand, we claim that the right hand side converges to zero. Indeed, since $\{\lambda_j\}_{j=1}^\infty$ is bounded, we deduce
\begin{equation}\label{Boundedness-SeqConfFact3}
\begin{split}
\left|\lambda_j\int_\Sigma \tilde \psi_j\psi u_{j}^{N-2}\;d\sigma_g\right| &\le C\|\psi\|_{L^\infty(\Sigma)} \int_\Sigma |\tilde \psi_j|u_{j}^{N-2}\;d\sigma_g \\ &\le \frac{C\|\psi\|_{L^\infty(\Sigma)}}{\|\psi_{j}\|_{L^2(\Sigma)}}\underbrace{\left(\int_\Sigma \psi_{j}^2u_{j}^{N-2}\;d\sigma_g\right)^{\frac{1}{2}}}_{=1}\underbrace{\left(\int_\Sigma u_{j}^{N-2}\;d\sigma_g\right)^{\frac{1}{2}}}_{\le 1} \\ & \to 0.
\end{split}
\end{equation}
Putting everything together, we have that for all $\psi\in C^\infty(\overline{M})$,
\begin{equation}\label{Boundedness-SeqConfFact4}
\int_M \langle\nabla_g \tilde \psi, \nabla_g \psi \rangle\;dv_g + 2c_nh_g\int_\Sigma \tilde \psi\psi \;d\sigma_g = 0. 
\end{equation}
Since $0\not\in \text{Spec}(\D_g)$, we conclude from (\ref{Boundedness-SeqConfFact4}) that $\tilde \psi_{|_\Sigma} = 0$. However, since $\|\tilde\psi_j\|_{L^2(\Sigma)} = 1$, we then have $\int_\Sigma \tilde \psi^2\;d\sigma_g = 1$ by strong convergence. This is a contradiction, thus $\{\psi_j\}_{j=1}^\infty$ is bounded in $L^2(\Sigma)$ and the proof is completed.
\end{proof}

We proceed with showing existence of generalized eigenfunctions and the simplicity of the first generalized eigenfunction. 

\begin{proposition}\label{ExistenceSimplicity}
If $u\in L^N_{> 0}(\Sigma)$, then 
\begin{enumerate}
\item (Existence) There exist $f_1$ and $\phi_2$ in $W^{1,2}(M)$ solving (\ref{EigenValEq}) in the sense of distributions with $\lambda=\lambda_1(u)$ and $\lambda=\lambda_2(u)$, respectively, and satisfying the following conditions:
\begin{equation}\label{EigenfunctionProp}
\int_\Sigma f_1^2u^{N-2}\;d\sigma_g = \int_\Sigma \phi_2^2u^{N-2}\;d\sigma_g = 1, \; \int_\Sigma f_1\phi_2u^{N-2}\;d\sigma_g = 0.
\end{equation}
\item (Simplicity) The zero set $f_1^{-1}(0)\cap \Sigma$ of the first eigenfunction $f_1$ restricted to $\Sigma$ has zero (boundary) measure, and so a second generalized eigenfunction $\phi_2$ changes sign on $\Sigma$. Moreover, $E_1(u)$ is one dimensional. 
\end{enumerate}
\end{proposition}

\begin{proof}
We start with the existence of $f_1$. Let $\{f_j\}_{j=1}^\infty\subset W^{1,2}(M)$ be a minimizing sequence for $\lambda_1(u)$, and assume that for each $j\ge 1$, $f_j\ge 0$ and $\int_\Sigma f_j^2u^{N-2}\;d\sigma_g =1$. It follows from Proposition \ref{finite} that the sequence $\{f_j\}_{j=1}^\infty$ is bounded in $W^{1,2}(M)$. By standard compactness arguments, there exists a function $f_1\in W^{1,2}(M)$ such that $f_j \rightharpoonup  f_1 \text{ in } W^{1,2}(M)$, $f_j\rightarrow  f_1 \text{ in } L^2(M)$, and $f_j \rightarrow  f_1 \text{ in } L^2(\Sigma)$. Therefore, 
\[
\begin{split}
\int_M|\nabla_g f_1|^2\;dv_g + 2c_n\int_\Sigma h_g f_1^2\;d\sigma_g &\le \liminf_{j\to\infty}\left(\int_M|\nabla_g f_j|^2\;dv_g + 2c_n \int_\Sigma h_g f_j^2\;d\sigma_g\right) \\ & =\liminf_{j\to\infty}\R_g^u(f_j)= \liminf_{j\to\infty}\lambda_j\\ &= \lambda_1(u).
\end{split}
\]

In order to finish the proof, we need to argue that $\R_g^u(f_1) = \lambda_1(u)$. Thus, what remains to be shown is
\begin{equation}\label{EigenfunctionUnit}
\int_\Sigma f_1^2u^{N-2}\;d\sigma_g=1.
\end{equation}
Using $\int_\Sigma f_j^2u^{N-2}\;d\sigma_g = 1$ for all $j\ge 1$, we estimate as follows:
\begin{equation}
\begin{split}
\left|\int_\Sigma f_1^2u^{N-2}\;d\sigma_g - 1\right| & = \left|\int_\Sigma( f_1^2 -  f_j^2)u^{N-2}\;d\sigma_g\right| \\ &\le \underbrace{\int_\Sigma| f_1^2- f_j^2|u^{N-2}\;d\sigma_g}_{=A_j}.
\end{split}
\end{equation}
For any $C>0$ sufficiently large, define $u_C:=\inf\{C,u\}$. Then
\begin{equation}\label{ExistenceEstimate}
\begin{split}
A_j &\le \int_\Sigma |f_1^2-f_j^2||u_C^{N-2} - u^{N-2}|\;d\sigma_g + \int_\Sigma|f_1^2-f_j^2|u_C^{N-2}\;d\sigma_g \\ &\le \left(\int_\Sigma|f_1^2-f_j^2|^{\frac{N}{2}}\;d\sigma_g\right)^{\frac{2}{N}}\left(\int_\Sigma |u_C^{N-2} - u^{N-2}|^{\frac{N}{N-2}}\;d\sigma_g\right)^{\frac{N-2}{N}} \\ &\hspace{.15in}+C^{N-2}\int_\Sigma|f_1^2-f_j^2|\;d\sigma_g.
\end{split}
\end{equation}
Since $\{f_j\}_{j=1}^\infty$ is bounded in $W^{1,2}(M)$, we conclude via the Sobolev trace embedding (\ref{STE}) that the first factor of the first term in the last two lines is bounded. The second factor goes to zero as $C\to \infty$ by the Dominated Convergence Theorem, whereas the last term goes to zero by strong convergence in $L^2(\Sigma)$. This proves that $A_j$ goes to zero by first letting $j\to\infty$, and then $C\to\infty$. Hence, $\R^u_g(f_1) = \lambda_1(u)$ and we have the existence of a first generalized eigenfunction. 

To prove the existence of a second generalized eigenfunction $\phi_2$ satisfying (\ref{EigenfunctionProp}), we start with a normalized minimizing sequence $\{\phi_j\}_{j=1}^\infty\subset W^{1,2}(M)$ for $\tilde \lambda_2(u) := \inf \R_g^u(\phi)$, where the infimum is being taken over all $\phi \in W^{1,2}(M)$ satisfying 
\begin{equation}
\int_\Sigma  f_1\phi u^{N-2}\;d\sigma_g = 0.
\end{equation} 
By Lemma \ref{Char2} below, $0>\lambda_2(u) = \tilde \lambda_2(u)$. With this, we conclude that this sequence is bounded in $W^{1,2}(M)$ thanks to Proposition \ref{finite}. From standard compactness arguments, we get convergence properties as for the sequence $\{f_j\}_{j=1}^\infty$. Denote the limiting function by $\phi_2$. To prove the orthogonality condition, we compute
\begin{equation}
\begin{split}
\left|\int_\Sigma f_1\phi_2u^{N-2}\;d\sigma_g - 0\right|^2 & =\left|\int_\Sigma f_1\phi_2u^{N-2}\;d\sigma_g - \int_\Sigma f_1\phi_ju^{N-2}\;d\sigma_g\right|^2 \\ & =\left|\int_\Sigma\left(f_1u^{\frac{N-2}{2}}\right)(\phi_2 - \phi_j)u^{\frac{N-2}{2}}\;d\sigma\right|^2 \\ &\le \int_\Sigma f_1^2u^{N-2}\;d\sigma_g\cdot\int_\Sigma|\phi_2-\phi_j|^2u^{N-2}\;d\sigma_g,
\end{split}
\end{equation}
and the estimate of the second factor is as before (see estimates for $A_j$). Moreover, the same truncation argument used for $f_1$ shows that $\int_\Sigma\phi_2^2u^{N-2}\,d\sigma_g=1$. This completes the proof of part (i). 

We now proceed with the proof of part (ii). Notice that, by construction, the function $f_1$ is nonnegative in $M$. By continuity of the trace operator and approximation, its restriction to $\Sigma$ is also nonnegative. We want to show that it is positive almost everywhere on $\Sigma$, i.e. the zero set of its trace has zero boundary measure. 

Recall that the background metric $g$ is scalar flat, and so $f_1$ is a $W^{1,2}(M)$-solution of the system
\begin{equation}\label{ES1}
\begin{cases}
\Delta_g f_1 &= 0\; (\text{in } M)\\
B_g(f_1) &= \lambda_1(u) f_1u^{N-2}\; (\text{on } \Sigma).
\end{cases}
\end{equation}
By the Harnack inequality (see Corollary 8.21 in \cite{Gilbarg}), we conclude that $f_1>0$ almost everywhere in $M$. For each $\epsilon > 0$, consider the function 
\begin{equation}
\psi_\epsilon:= \frac{1}{f_1+\epsilon}.
\end{equation}
Since $\psi_\epsilon \in W^{1,2}(M)$, we can use it as a test function in (\ref{ES1}) and get
\begin{equation}
\lambda_1(u)\int_\Sigma \frac{f_1}{f_1+\epsilon}u^{N-2} \;d\sigma_g = -\int_M \frac{|\nabla_gf_1|^2}{(f_1+\epsilon)^2}\;dv_g +2c_n\int_\Sigma h_g\frac{f_1}{f_1+\epsilon}\;d\sigma_g.
\end{equation}
Since $h_g<0$, we derive
\begin{equation}
\begin{split}
\int_M \frac{|\nabla_gf_1|^2}{(f_1+\epsilon)^2}\;dv_g &\le (-\lambda_1(u))\int_\Sigma \frac{f_1}{f_1+\epsilon} u^{N-2}\;d\sigma_g\le (-\lambda_1(u)) \int_\Sigma u^{N-2}\;d\sigma_g,
\end{split}
\end{equation}
and an application of Fatou's lemma gives $|\nabla_g \log f_1|\in L^2(M)$. 

Finally, consider the function 
\begin{equation}
f_\epsilon = \frac{f_1}{f_1+\epsilon}.
\end{equation}
Using that $f_\epsilon >0$ a.e. in $M$, it follows that $f_\epsilon \to 1$ in $L^2(M)$. Also, $\nabla_g f_\epsilon \to 0$ a.e. in $M$. Since $|\nabla_g f_\epsilon|^2 \le |\nabla_g \log f_1|^2\in L^1(M)$, an application of the Dominated Convergence Theorem yields that $f_\epsilon \to 1$ in $W^{1,2}(M)$. Therefore, $f_\epsilon \to 1$ a.e. on $\Sigma$, thus $f_1$ is positive a.e. on $\Sigma$. The simplicity of $\lambda_1(u)$ now follows from standard arguments; see proof of Proposition 2.4, (ii) in \cite{Gursky}.
\end{proof}

\begin{lemma}\label{Char2}
If $u\in L^N_{> 0}(\Sigma)$, then 
\begin{equation} \label{char}
\lambda_2(u) = \inf \R_g^u(\phi),
\end{equation}
where the infimum is being taken over all $\phi \in W^{1,2}(M)$ with $\;\phi|_\Sigma \not \equiv 0$ and satisfying $\int_\Sigma \phi f_1u^{N-2}\;d\sigma_g = 0$. Here $f_1$ is the first eigenfunction associated to $\lambda_1(u)>-\infty$.  
\end{lemma}

\begin{proof}
As in the proof of Proposition \ref{ExistenceSimplicity}, we denote the right hand side of (\ref{char}) by $\tilde \lambda_2(u)$. For any two dimensional subspace $S_2\in Gr_2^u(W^{1,2}(M))$, there is always a function $\tilde \phi\in S_2$ satisfying
\begin{equation}
\int_\Sigma \tilde \phi f_1u^{N-2}\;d\sigma_g = 0.
\end{equation}
Therefore,
\begin{equation}
\sup_{\phi\in S_2,\;\phi|_\Sigma \not \equiv 0}\R_g^u(\phi) \ge \R_g^u(\tilde \phi)\ge \tilde \lambda_2(u),
\end{equation}
and thus $\lambda_2(u)\ge \tilde \lambda_2(u)$.

On the other hand, let $\tilde \phi$ be any test function for $\tilde \lambda_2(u)$, and consider $S_2:=\text{span}\{\tilde\phi, f_1\}$. The orthogonality condition and $f_1>0$ a.e. on $\Sigma$ imply that $S_2\in Gr^u_2(W^{1,2}(M))$. We claim that  $\R_g^u(\tilde \phi) =\sup_{\phi\in S_2,\;\phi|_\Sigma \not \equiv 0}\R_g^u(\phi)$. Indeed, for any $a,b\in\mathbb{R}$, $\int_\Sigma \tilde f_1u^{N-2}\;d\sigma_g =0$ yields
\begin{equation}
\begin{split}
\R_g^u(a\tilde \phi +bf_1) &= \frac{a^2\R_g^u(\tilde \phi)\int_\Sigma \tilde \phi^2 u^{N-2}\;d\sigma_g + b^2\lambda_1(u)\int_\Sigma f_1^2u^{N-2}\;d\sigma_g}{a^2\int_\Sigma\tilde\phi^2u^{N-2}\;d\sigma_g + b^2\int_\Sigma f_1^2u^{N-2}\;d\sigma_g} \\ & \hspace{.15in}+\frac{2ab\lambda_1(u)\int_\Sigma\tilde\phi f_1\;u^{N-2}\;d\sigma_g}{a^2\int_\Sigma\tilde\phi^2u^{N-2}\;d\sigma_g + b^2\int_\Sigma f_1^2u^{N-2}\;d\sigma_g} \\ &\le \frac{a^2\R_g^u(\tilde \phi)\int_\Sigma \tilde \phi^2 u^{N-2}\;d\sigma_g + b^2\R_g^u(\tilde\phi)\int_\Sigma f_1^2u^{N-2}\;d\sigma_g}{a^2\int_\Sigma\tilde\phi^2u^{N-2}\;d\sigma_g + b^2\int_\Sigma f_1^2u^{N-2}\;d\sigma_g} \\ &=\R_g^u(\tilde\phi),
\end{split}
\end{equation}
and the claim follows because $\tilde \phi\in S_2$. Since $\sup_{\phi\in S_2, \;\phi|_\Sigma \not \equiv 0}\R_g^u(\phi)\ge \lambda_2(u)$, after taking the infimum over all such $\tilde \phi$ we conclude that $\tilde\lambda_2(u)\ge\lambda_2(u)$. This completes the proof.
\end{proof}

We finish this section with an important consequence of Proposition \ref{NegativeEigen}. For any generalized conformal factor $u$ in $L^N_{>0}(\Sigma)$, the associated second eigenspace $E_2(u)$, which is nontrivial by Proposition \ref{ExistenceSimplicity}, is finite dimensional. In fact, since $\text{dim }E_1(u) = 1$ by Proposition \ref{ExistenceSimplicity}, we have that $\text{dim }E_2(u)\le N([g]) - 1$. More generally, this is the case for any eigenspace associated with any negative eigenvalue. We record this as a corollary: 

\begin{corollary}\label{E2-FiniteDim}
If $u\in L^N_{>0}(\Sigma)$, then the associated second eigenspace $E_2(u)$ is finite dimensional. More precisely, $\text{dim }E_2(u) \le N([g]) - 1$.
\end{corollary}

\subsection{$L^\infty$ estimates for Generalized Eigenfunctions}
Recall that, thanks to the Sobolev trace embedding theorem (\ref{STE}), if a function is in $W^{1,2}(M)$, then it is in $L^N(\Sigma)$. The following improvement for integrability along the boundary for weak solutions of (\ref{EigenValEq}) is a generalization of a result in \cite{Ammann}: \textit{let $u\in L^N_{\ge 0}(\Sigma)$ and $\phi \in W^{1,2}(M)$. If $\phi$ solves (\ref{EigenValEq}) in the sense of distributions, then $\phi \in L^{N+\epsilon}(\Sigma)$ for some $\epsilon>0$}. The techniques for proving it are very similar although in this case it involves having to control boundary integrals.   However, this is not enough for our purposes, thus we omit the proof. In the next proposition, we prove via Moser iteration techniques a stronger result in the case of generalized eigenfunctions associated with negative eigenvalues and a potential $u^{N-2}$ which is not assumed to be in $L^\infty(\Sigma)$.

\begin{proposition}\label{Estimates}
Let $u\in L^N_{\ge 0}(\Sigma)$ and $\phi\in W^{1,2}(M)$. Assume $\phi$ is a weak $W^{1,2}(M)$-solution of 
\begin{equation}\label{SEigenReg1}
\begin{cases}
L_g(\phi) &= 0 \;(\text{in }M)\\
B_g(\phi) &= -\phi u^{N-2} \; (\text{on }\Sigma).
\end{cases}
\end{equation}
Then 
\begin{enumerate}[(i)]
\item $\phi\in L^p(M)\cap L^p(\Sigma)$ for each finite $p\ge 1$, and
\item $\phi \in L^\infty(M)\cap L^\infty(\Sigma)$.
\end{enumerate}
\end{proposition}

\begin{proof}

The weak formulation of (\ref{SEigenReg1}) implies that 
\begin{equation}\label{SEigenReg2}
-\int_\Sigma \psi \phi u^{N-2}\;d\sigma_g = \int_M \la \nabla_g\psi,\nabla_g\phi \ra\;dv_g + 2c_n\int_\Sigma h_g\psi \phi\;d\sigma_g
\end{equation}
for all $\psi \in W^{1,2}(M)$. For $h>0$, define $\phi_h:= \inf\{|\phi|, h\}$, and take as a test function $\psi = \phi \phi_h^{2k}$, where $k> 0$. Plugging $\psi$ back into (\ref{SEigenReg2}) and rearranging gives
\begin{equation}\label{SEigenReg3}
\begin{split}
\int_M \phi_h^{2k}|\nabla_g\phi|^2\;dv_g &+ \underbrace{2k\int_M\phi_h^{2k-1}\phi\la\nabla_g\phi_h,\nabla_g \phi\ra\;dv_g}_{:= A_1}\\ &= -\int_\Sigma \phi_h^{2k}\phi^2u^{N-2}\;d\sigma_g - 2c_n\int_\Sigma h_g\phi_h^{2k}\phi^2\;d\sigma_g.
\end{split}
\end{equation}

The right-hand side of (\ref{SEigenReg3}) can be estimated by
\begin{equation}\label{SEigenReg4}
-\int_\Sigma \phi_h^{2k}\phi^2u^{N-2}\;d\sigma_g - 2c_n\int_\Sigma h_g\phi_h^{2k}\phi^2\;d\sigma_g \le 2c_n(-h_g)\int_\Sigma \phi_h^{2k}\phi^2\;d\sigma_g.
\end{equation}
Also, notice that $A_1$ on the left-hand side equals
\begin{equation}\label{SEigenReg4.1}
A_1 = 2k\int_{\{|\phi|\le h\}} \phi_h^{2k}|\nabla_g\phi|^2\;dv_g,
\end{equation}
where $\{|\phi|\le h\}: = \{x\in M: |\phi(x)|\le h\}$. The set $\{|\phi|>h\}$ is defined similarly. From (\ref{SEigenReg4}) and (\ref{SEigenReg4.1}) we then obtain
\begin{equation}\label{SEigenReg5}
\underbrace{\int_M \phi_h^{2k}|\nabla_g\phi|^2\;dv_g + 2k\int_{\{|\phi|\le h\}} \phi_h^{2k}|\nabla_g\phi|^2\;dv_g}_{:= A_2} \le C(g)\int_\Sigma\phi_h^{2k}\phi^2\;d\sigma_g,
\end{equation}
where $C(g) = 2c_n(-h_g)>0$. To estimate $A_2$, we proceed as follows:
\begin{equation}
\begin{split}
A_2 &= \int_{\{|\phi|>h\}}\phi_h^{2k}|\nabla_g\phi|^2\;dv_g + (2k+1)\int_{\{|\phi|\le h\}} \phi_h^{2k}|\nabla_g\phi|^2\;dv_g \\ &\ge \frac{2k+1}{(k+1)^2}\int_{\{|\phi|>h\}} \phi_h^{2k}|\nabla_g\phi|^2\;dv_g + (2k+1)\int_{\{|\phi|\le h\}} \phi_h^{2k}|\nabla_g\phi|^2\;dv_g
\end{split}
\end{equation}
On the other hand, using
\begin{equation}
|\nabla_g(\phi\phi_h^k)|^2 =
\begin{cases}
 (k+1)^2\phi_h^{2k}|\nabla_g\phi|^2 & \{|\phi|\le h\}\\
 \phi_h^{2k}|\nabla_g\phi|^2 & \{|\phi|>h\},
\end{cases}
\end{equation}
we obtain
\[
\int_M|\nabla_g(\phi\phi_h^k)|^2\;dv_g = \int_{\{|\phi|>h\}} \phi_h^{2k}|\nabla_g\phi|^2\; dv_g + (k+1)^2\int_{\{|\phi|\le h\}} \phi_h^{2k}|\nabla_g\phi|^2\;dv_g
\]
Thus, 
\begin{equation}\label{SEigenReg6}
A_2 \ge \frac{2k+1}{(k+1)^2}\int_M|\nabla_g(\phi\phi_h^k)|^2\;dv_g.
\end{equation}
Putting together (\ref{SEigenReg5}) and (\ref{SEigenReg6}) gives 
\begin{equation}
\frac{2k+1}{(k+1)^2} \int_M|\nabla_g(\phi\phi_h^k)|^2\;dv_g \le C(g)\int_\Sigma (\phi\phi_h^k)^2\;d\sigma_g,
\end{equation}
that is,
\begin{equation}\label{SEigenReg7}
\frac{2k+1}{(k+1)^2} \|\phi\phi_h^k\|^2_{W^{1,2}(M)} \le C(g)\int_\Sigma (\phi\phi_h^k)^2\;d\sigma_g + \frac{2k+1}{(k+1)^2}\int_M (\phi\phi_h^k)^2\;dv_g.
\end{equation}

\vspace{.1in}

\noindent{\bf Interior Estimates: $\phi \in L^p(M),  1\le p\le \infty$.} We focus on the right-hand side of (\ref{SEigenReg7}). By a result of Winkert \cite{WinkertBoundedness} (see also Proposition 2.1 in \cite{Marino}), there exist constants $c_1>0$ and $c_2>0$ such that, for any $\epsilon>0$\footnote{We point out that the constants $c_1$ and $c_2$ from Proposition 2.1 in \cite{Marino}, which traces back to \cite{WinkertBoundedness}, are independent of $\epsilon$. This is used here.},  
\begin{equation}
\|\phi\phi_h^k\|^2_{L^2(\Sigma)} \le \epsilon \|\phi\phi_h^k\|^2_{W^{1,2}(M)} + c_1\epsilon^{-c_2}\|\phi\phi_h^k\|^2_{L^2(M)}.
\end{equation}
Taking $\epsilon := C(g)^{-1}(2k+1)(k+1)^{-2}2^{-1}$, from (\ref{SEigenReg7}) we deduce
\begin{equation}
\|\phi\phi_h^k\|_{W^{1,2}(M)}^2\le C(g,k)\|\phi\phi_h^k\|^2_{L^2(M)},
\end{equation}
where 
\begin{equation}
C(g,k) = 2 + C(g)^{c_2+1}c_1\left(\frac{2(k+1)^2}{2k+1}\right)^{c_2+1}.
\end{equation}
By enlarging $C(g,k)$ if necessary, we can assume without loss of generality that $C(g,k) \le C(g)(1+k^{1+c_2})$, for a possibly different constant $C(g)$ independent of $k$. 

We use $2^*$ to denote $\frac{2n}{n-2}$. An application of the Sobolev embedding theorem implies
\begin{equation}\label{SEigenReg8}
\begin{split}
\|\phi\phi_h^k\|_{L^{2^*}(M)} &\le C\|\phi\phi_h^k\|_{W^{1,2}(M)} \le C\cdot C(g)^{\frac{1}{2}}(1+k^{1+c_2})^{\frac{1}{2}}\|\phi\phi_h^k\|_{L^2(M)} \\ &\le C\cdot C(g)^{\frac{1}{2}}(1+k^{1+c_2})^{\frac{1}{2}}\|\phi\|_{L^{2(k+1)}(M)}^{k+1}.
\end{split}
\end{equation}
Letting $h\to \infty$ and using Fatou's Lemma gives
\begin{equation}
\|\phi\|_{L^{p\frac{n}{n-2}}(M)} \le C(g)^{\frac{1}{p}} (1+k^{1+c_2})^{\frac{1}{p}} \|\phi\|_{L^{p}(M)},
\end{equation}
where $p = 2(k+1)$. Select $k_0>0$ such that $2(k_0+1) = 2^*$, and pick $k_j$, $j\in \mathbb{N}$, such that $2(k_j+1) = \left(\frac{n}{n-2}\right)^j2^*$. A standard iteration scheme yields
\begin{equation}\label{SEigenReg9}
\|\phi\|_{L^{p_j\frac{n}{n-2}}(M)}\le C \|\phi\|_{L^{2^*}(M)},
\end{equation}
where $C = C(n,g,c_2)$. Since $p_j\to \infty$, this also gives the $L^\infty(M)$ estimates. 

\vspace{.1in}
\noindent{\bf Boundary Estimates: $\phi \in L^p(\Sigma), 1\le p\le\infty$.}
It is enough to show that $\phi\in L^\infty(\Sigma)$. Since $\phi\in W^{1,2}(M)\cap L^\infty(M)$ thanks to the previous discussion, this case is a consequence of Proposition 2.4 in \cite{Marino}.
\end{proof}

\begin{remark}\label{GenEigenBounds}
Notice that Proposition \ref{Estimates} remains true if $\phi$ is any generalized eigenfunction associated to a negative eigenvalue. That is, Proposition \ref{Estimates} holds for any weak $W^{1,2}(M)$ solution of 
\begin{equation}
\begin{cases}
L_g(\phi) &= 0 \;(\text{in }M)\\
B_g(\phi) &= \lambda_k(u)\phi u^{N-2} \; (\text{on }\Sigma)
\end{cases}
\end{equation}
with $\lambda_k(u)<0$.
\end{remark}

\begin{remark}\label{RMoserBound}
Suppose $\{\phi_i\}_{i=1}^\infty$ is a sequence of weak solutions to (\ref{SEigenReg1}) which is uniformly bounded in $W^{1,2}(M)$. By the Sobolev embedding theorem, the sequence is also uniformly bounded in $L^{2^*}(M)$. From (\ref{SEigenReg9}) we then conclude that the sequence is uniformly bounded in any $L^p(M)$. Furthermore, after taking limits in (\ref{SEigenReg9}), we can also conclude that the sequence is uniformly bounded in $L^\infty(M)$. Finally, an application of Proposition 2.4 in \cite{Marino} yields uniform bounds in $L^\infty(\Sigma)$.
\end{remark}

\begin{remark}
The ideas for the proof of Proposition \ref{Estimates} follow closely the approach of Marino and Winkert in \cite{Marino} (see Theorem 3.1). However, their result does not apply immediately since our weight function $u$ is not a priori in $L^\infty(\Sigma)$. 
\end{remark}

\section{First Eigenvalue: Proof of Theorem \ref{TFirstEigenvalue}}\label{FirstEigenvalue}

Assume that we are given a smooth conformal metric $g_u= u^{\frac{4}{n-2}}g\in[g]$ ($u\in C^\infty(\overline M)$, $u>0$ on $\overline M$), and that $[g]$ is a conformal class for which inequality (\ref{ExistenceCond}) holds. Thanks to Escobar's work \cite{Escobar2}, a minimizer exists and the associated metric is scalar flat and has constant mean curvature on the boundary. Let us assume that our background metric $g$ is such metric, and that it has been normalized to have unit boundary volume. Therefore, $Q(\overline M,\Sigma, [g]) = 2c_nh_{g}$. Denote by $E_g(\phi)$ the numerator in the Rayleigh Quotient $\R_g^u$:
\begin{equation}
\begin{split}
E_g(\phi) &:= \int_M (|\nabla_g\phi|^2 +c_nR_g\phi^2)\;dv_g + 2c_n\int_\Sigma h_g\phi^2\;d\sigma_g \\ &= \int_M |\nabla_g\phi|^2 \;dv_g + 2c_nh_g\int_\Sigma \phi^2\;d\sigma_g.
\end{split}
\end{equation}
The sign of $h_g$ is determined by the sign of $Q(\overline{M},\Sigma,[g])$. 

 \subsection{The $Q(\overline{M},\Sigma,[g]) < 0$ case} We start by pointing out that, in this case, we have uniqueness of extremals for $Q(\overline M, \Sigma, [g])$, and so a scalar flat metric with constant mean curvature on the boundary is automatically a minimizer. This follows from a maximum principle argument; see Proposition 2.6 in \cite{Hamanaka}.
 
 Recall that $\lambda_1(\D_{g}) < 0$ as its sign agrees with that of $Q(\overline{M},\Sigma,[g])$. The variational characterization (\ref{VarCar2}) implies that we can find a function $\phi\in C^\infty(\overline M)$ such that $E_g(\phi)<0$ and $\phi\not\equiv0$ on $\Sigma$. For these functions, Hölder's inequality implies
\begin{equation}\label{FirstEigenvalue1}
\R_g^u(\phi) = \frac{E_g(\phi)}{\int_\Sigma \phi^2u^{N-2}\;d\sigma_g} \le \frac{E_g(\phi)}{\left(\int_\Sigma \phi^{N}\;d\sigma_g\right)^{\frac{2}{N}}\left(\int_\Sigma u^{N}\;d\sigma_g\right)^{\frac{N-2}{N}}}.
\end{equation}
Therefore,
\begin{equation}\label{FirstEigenvalueQ}
\lambda_1(u)\text{Vol}(\Sigma,d\sigma_{g_u})^{\frac{N-2}{N}} \le Q(\overline{M},\Sigma,[g]) 
\end{equation}
follows after taking the infimum over all those $\phi$'s. Notice that equality occurs for the metric $g$, i.e., $u\equiv 1$ maximizes the normalized first eigenvalue functional. Indeed, if we take $\phi\equiv 1>0$, then $B_{g}(1) = 2c_nh_{g}$, implying that $\lambda_1(1) = 2c_n h_{g} = Q(\overline M,\Sigma, [g])$.

We claim that equality in (\ref{FirstEigenvalueQ}) occurs if and only if, up to a conformal change of metric in the interior, $u$ is a positive constant on $\Sigma$. This means that $g_u|_{T\Sigma}$ is homothetic to $g|_{T\Sigma}$. First, let us assume that such $u$ is given. By scale invariance, we can assume that $\text{Vol}(\Sigma,d\sigma_{g_u}) = 1$, and so $u\equiv 1$ on $\Sigma$. Then $B_{g_u}(u^{-1}) = u^{1-N}B_g(1) = 2c_nh_g$, thus $u^{-1}$ is a first eigenfunction with $\lambda_1(u) = 2c_nh_g = Q(\overline{M},\Sigma,[g])$. 

On the other hand, assume that equality in (\ref{FirstEigenvalueQ}) occurs for some conformal metric $g_u=u^{\frac{4}{n-2}}g\in[g]$. By scale invariance, we can assume that $\text{Vol}(\Sigma,d\sigma_{g_u}) = 1$, thus $\lambda_1(u)= Q(\overline M, \Sigma, [g]) = 2c_nh_g$.  Moreover, using (\ref{FirstEigenvalue1}), we get
\begin{equation}
\lambda_1(u)\le \R_g^u(1) = \frac{E_g(1)}{\int_\Sigma u^{N-2}\;d\sigma_g}\le \frac{E_g(1)}{\left(\int_\Sigma u^N\;d\sigma_g\right)^{\frac{N-2}{N}}} = 2c_nh_g= \lambda_1(u).
\end{equation}
The equality case in H\"older's inequality then implies that $u$ is constant on $\Sigma$. This finishes the proof. 

As a concluding remark, we point out that (\ref{FirstEigenvalue1}) shows that if $Q(\overline{M},\Sigma,[g]) = -\infty$, then so is $\lambda_1(u)\text{Vol}(\Sigma,d\sigma_{g_u})^{\frac{N-2}{N}}$ for any given smooth positive $u$.

\subsection{The $Q(\overline{M},\Sigma,[g])>0$ case} In this case, we have that $\lambda_1(\D_{g})>0$, and we assume the background metric $g$ is a unit boundary volume minimizer of $Q(\overline{M},\Sigma,[g])$. Since $E_g(\phi) >0$ for all $\phi\in C^{\infty}(\overline{M})$ with $\phi \not \equiv 0$ on $\Sigma$, Hölder's inequality yields
\begin{equation}\label{FirstEigenvalue2}
\R_g^u(\phi) = \frac{E_g(\phi)}{\int_\Sigma \phi^2u^{N-2}\;d\sigma_g} \ge \frac{E_g(\phi)}{\left(\int_\Sigma \phi^{N}\;d\sigma_g\right)^{\frac{2}{N}}\left(\int_\Sigma u^{N}\;d\sigma_g\right)^{\frac{N-2}{N}}}.
\end{equation}
and a similar analysis leads to
\begin{equation}\label{FirstEigenvalueQ1}
\lambda_1(\D_{g_u})\text{Vol}(\Sigma,d\sigma_{g_u})^{\frac{N-2}{N}}\ge Q(\overline{M},\Sigma,[g]). 
\end{equation}
As before, equality occurs if we take $u\equiv 1$. 

Assume equality holds in (\ref{FirstEigenvalueQ1}) for some conformal metric $g_u$, and assume it has unit boundary volume. If $\phi_1>0$ is a first eigenfunction associated to $\lambda_1(u) = 2c_nh_g$, then
\begin{equation}
\begin{split}
\lambda_1(u) &= \frac{E_{g}(\phi_1)}{\int_\Sigma \phi_1^2u^{N-2}\;d\sigma_g} = Q(\overline M, \Sigma, [g]) \\ &\le \frac{E_{g}(\phi_1)}{\left(\int_\Sigma \phi_1^N\;d\sigma_g\right)^{\frac{2}{N}}} \le \frac{E_{g}(\phi_1)}{\int_\Sigma \phi_1^2u^{N-2}\;d\sigma_g} = \lambda_1(u).
\end{split}
\end{equation} 
This implies that $\phi_1>0$ minimizes $Q(\overline M, \Sigma, [g])$, and so $g_{\phi_1}$ is a scalar flat metric with constant mean curvature, and it implies that $u = C\phi_1$ on $\Sigma$, for some $C>0$. Conversely, if $\phi_1$ is a positive first eigenfunction for $g_u$, $g_{\phi_1}$ minimizes $Q(\overline M, \Sigma, [g])$ and $u = C\phi_1$ on $\Sigma$ for some $C>0$, then 
\begin{equation}
\begin{split}
\lambda_1(u) &= \frac{E_g(\phi_1)}{\int_\Sigma \phi_1^2u^{N-2}\;d\sigma_g} = C^{-2}E_g(\phi_1)\left(\int_\Sigma u^N\;d\sigma_g\right)^{-1} \\ &= Q(\overline M,\Sigma, [g])\text{Vol}(\Sigma, d\sigma_{g_u})^{\frac{2-N}{N}}.
\end{split}
\end{equation}
This finishes the proof.

\subsection{The $Q(\overline M,\Sigma, [g]) = 0$ case.}
In this case, our background metric $g$ satisfies both $R_g = 0$ and $h_g = 0$. If $g_u\in [g]$, then the conformal properties imply that $B_{g_u} (u^{-1})= u^{1-N}B_g(1) = 0$ and so $\lambda_1(u) = 0$. Since $g_u\in [g]$ is arbitrary, the result follows. 
\section{First Variation Formulas}\label{FirstVariationFormulas}

One of the main issues in extremal eigenvalue problems arises from the non-differentiability of eigenvalue functionals of the form $g\in \mathcal{A}\longmapsto \lambda(g)$, where $\mathcal{A}$ is a subset of the space of Riemannian metrics on the manifold. However, the one-sided derivatives exist along certain deformations, and Euler--Lagrange equations can be obtained via classical Hahn--Banach separation theorems; see for instance \cite{Soufi,Soufi2} for discussion involving Laplacian eigenvalues. 

The main purpose of this section is to introduce a suitable family of deformations, and to derive first variation formulas. We consider deformations of the form $u_t:= u(1+ th)$, where $h\in L^\infty(\Sigma)$. In light of Lemma \ref{TestFunctions}, we require $t\in (-\delta,\delta)$, where $\delta >0$ is small enough such that $|1+th| \ge 1 - \delta\|h\|_{L^\infty(\Sigma,d\sigma)}>0$. With this, $Gr_k^u(W^{1,2}(M))=Gr_k^{u_t}(W^{1,2}(M))$ for all $t\in (-\delta,\delta)$ since the positive sets of $u$ and $u_t$ agree. Under these assumptions, the function $t\longmapsto \lambda_k(u_t)$ is continuous at $t=0$ for $k=1,2$. 

\begin{proposition}\label{Continuity}
Let $u\in L^N_{> 0}(\Sigma)$. Then the following holds:
\begin{equation}
\lim_{t\to0}\lambda_1(u_t) = \lambda_1(u),
\end{equation}
and 
\begin{equation}
\lim_{t\to 0}\lambda_2(u_t) = \lambda_2(u).
\end{equation}
\end{proposition}

\begin{proof}
The proof follows from Proposition \ref{NegativeEigen} (the zero sets of $u$ and $u_t$ coincide), the inequality
\begin{equation}
\begin{split}
(1-|t|\|h\|_\infty)^{N-2}\int_\Sigma \phi^2u^{N-2}\;d\sigma_g & \le \int_\Sigma \phi^2 u^{N-2}_t\;d\sigma_g\\ &\le (1+|t|\|h\|_\infty)^{N-2}\int_\Sigma u^{N-2}\phi^2\;d\sigma_g,
\end{split}
\end{equation}
and the variational characterization of generalized eigenvalues.
\end{proof}

For any generating function $h\in L^\infty(\Sigma)$, we define  $L_h(\phi, u)$ for a function $\phi$ in $E_2(u)$ by 
\begin{equation}
L_h(\phi, u) = -(N-2)\lambda_2(u)\frac{\int_\Sigma h\phi^2u^{N-2}\;d\sigma_g}{\int_\Sigma \phi^2u^{N-2}\;d\sigma_g}.
\end{equation}
One of the key ingredients in our argument is Proposition \ref{OSD} below. Its proof is omitted because the arguments are quite similar to those of the closed case discussed in \cite{Gursky}, Section 3 (see also \cite{Kokarev}).

\begin{proposition}\label{OSD}
The one-sided derivatives of the function $t\mapsto \lambda_2(u_t)$ exist, and they are given by
\begin{equation}
\frac{d}{dt}\lambda_2(u_t)_{|_{t=0^+}} = \inf_{\phi\in E_2(u)\setminus\{0\}} L_h(\phi,u),
\end{equation}
\begin{equation}
\frac{d}{dt}\lambda_2(u_t)_{|_{t=0^-}} = \sup_{\phi\in E_2(u)\setminus\{0\}} L_h(\phi,u).
\end{equation}
\end{proposition}

\section{Regularized Functional $F_\epsilon$}\label{eRegularization}

 For each $\epsilon>0$, we set
\begin{equation}
\D_\epsilon := \{u\in L^N_{\ge 0}: u^{-\epsilon}\in L^1(\Sigma)\}\subseteq L^N_{>0}(\Sigma),
\end{equation}
and define the scale-invariant functional $F_{2,\epsilon}: \D_\epsilon \to \mathbb{R}$ by
\begin{equation}
F_{2,\epsilon}(u) = \lambda_2(u)\left(\int_\Sigma u^N\;d\sigma_g\right)^{\frac{N-2}{N}} - \left(\int_\Sigma u^{-\epsilon}\;d\sigma_g\right)\left(\int_\Sigma u^N\;d\sigma_g\right)^{\frac{\epsilon}{N}}.
\end{equation}
Note that, by Proposition \ref{finite}, $F_{2,\epsilon}$ is well-defined on all of $\D_\epsilon$. Moreover, $F_{2,\epsilon}(u)\le 0$ for all $u\in \D_\epsilon$. As a straightforward consequence of Proposition \ref{OSD}, we obtain:

\begin{proposition}\label{FirstVarFor}
Suppose $u\in\D_\epsilon$ with
\begin{equation}\label{unitnorm}
\int_\Sigma u^N\;d\sigma_g = 1,
\end{equation}
and for any $h\in L^\infty(\Sigma)$ consider 
\begin{equation}\label{deformation}
u_t = u(1+th)
\end{equation}
for $t\in\mathbb{R}$ small. Then the one-sided derivatives of the function $t\longmapsto F_{2,\epsilon}(u_t)$ at $t=0$ exist, and are given by 
\begin{equation}
\begin{split}
\frac{d}{dt}F_{2,\epsilon}(u_t)|_{t=0^+} & = (N-2)\lambda_2(u)\Bigg{\{}\gamma_{1,\epsilon}\int_\Sigma hu^N\;d\sigma_g\\ &\hspace{.2in}- \inf_{\phi\in(E_2(u))_1}\int_\Sigma h\phi^2u^{N-2}\;d\sigma_g - \gamma_{2,\epsilon}\int_\Sigma hu^{-\epsilon}\;d\sigma_g\Bigg{\}}
\end{split}
\end{equation}
and
\begin{equation}
\begin{split}
\frac{d}{dt}F_{2,\epsilon}(u_t)|_{t=0^-} & = (N-2)\lambda_2(u)\Bigg{\{}\gamma_{1,\epsilon}\int_\Sigma hu^N\;d\sigma_g\\&\hspace{.2in} - \sup_{\phi\in(E_2(u))_1}\int_\Sigma h\phi^2u^{N-2}\;d\sigma_g - \gamma_{2,\epsilon}\int_\Sigma hu^{-\epsilon}\;d\sigma_g\Bigg{\}}
\end{split}
\end{equation}
Here 
\begin{equation}\label{gamma2}
\gamma_{1,\epsilon}:= 1-(N-2)^{-1}\frac{\epsilon}{\lambda_2(u)}\int_\Sigma u^{-\epsilon}\;d\sigma_g >1,
\end{equation}
\begin{equation}\label{gamma2}
\gamma_{2,\epsilon}:= (N-2)^{-1}\frac{\epsilon}{|\lambda_2(u)|} >0,
\end{equation}
and $(E_2(u))_1$ is the space of second generalized eigenfunctions with respect to $u$ satisfying $\int_\Sigma \phi^2 u^{N-2}\;d\sigma_g = 1$.
\end{proposition}

\begin{definition}\label{DefExtremal}
We say that $u\in \D_\epsilon$ is \textit{extremal} for the normalized functional $F_{2,\epsilon}$ if the one-sided derivatives of $t\longmapsto F_{2,\epsilon}(u_t)= F_{2,\epsilon}(u(1+th))$ at $t=0$ satisfy
\begin{equation}
\frac{d}{dt}F_{2,\epsilon}(u_t)_{t=0^+}\times \frac{d}{dt}F_{2,\epsilon}(u_t)_{t=0^-} \le 0
\end{equation}
for all generating functions $h\in L^\infty(\Sigma)$. If there is a $u_\epsilon\in \D_\epsilon$ such that 
\begin{equation}
F_{2,\epsilon}(u_\epsilon)\ge F_{2,\epsilon}(u)
\end{equation}
for all $u\in \D_{\epsilon}$, then we say $u_\epsilon$ is \textit{maximal} for $F_{2,\epsilon}$. 
\end{definition}

Notice that if a function $u\in \D_\epsilon$ is maximal, then 
\begin{equation}
\frac{d}{dt}F_{2,\epsilon}(u_t)_{t=0^+}\le 0 \le \frac{d}{dt}F_{2,\epsilon}(u_t)_{t=0^-}
\end{equation}
for all $h\in L^\infty(\Sigma)$, thus any maximal conformal factor is extremal. Also, our definition of extremal coincides with the one given by El Soufi-Ilias in \cite{Soufi}, which is equivalent to Nadirashvili's one in \cite{Nadirashvili}. 

In the next proposition, we show the existence of a maximizer $u_\epsilon \in \D_\epsilon$ for $F_{2,\epsilon}$, for each $\epsilon>0$.

\begin{proposition}\label{existence}
For each $\epsilon>0$, there is a $u_\epsilon\in\D_\epsilon$, normalized as in (\ref{unitnorm}), that is maximal for $F_{2,\epsilon}$. Moreover, there is a constant $C=C(g)$ independent of $\epsilon>0$ such that 
\begin{equation}\label{existence1}
\int_\Sigma u_\epsilon^{-\epsilon}\;d\sigma_g\le C.
\end{equation}
\end{proposition}

\begin{proof}
Fix $\epsilon>0$, and let $\{u_j\}_{j=1}^\infty\subset \D_\epsilon$ be a maximizing sequence for $F_{2,\epsilon}$: 
\begin{equation}
F_{2,\epsilon}(u_j) \to \alpha:=\sup_{u\in \D_\epsilon} F_{2,\epsilon}(u); \text{ and }\int_\Sigma u_j^N\;d\sigma_g=1.
\end{equation}
If we set $v_j = u_j^{N-2} = u_j^{\frac{2}{n-2}}$, then
\begin{equation}
\int_\Sigma v_j^{n-1}\;d\sigma_g = \int_\Sigma u_j^N\;d\sigma_g = 1,
\end{equation}
and thus $\{v_j\}_{j=1}^\infty$ is bounded in $L^{n-1}(\Sigma)$. By weak compactness, up to the extraction of a subsequence, there exists a $v\in L^{n-1}(\Sigma)$ such that $v_j \rightharpoonup v$ in $L^{n-1}(\Sigma)$. Therefore,  $v\ge 0$ a.e. and  $\int_\Sigma v^{n-1}\;d\sigma_g \le 1$. We set $u_\epsilon := v^{\frac{1}{N-2}}$, thus
\begin{equation}\label{existence2}
\int_\Sigma u_\epsilon^N\;d\sigma_g \le 1.
\end{equation}

We claim that $u_\epsilon$ is maximal for $F_{2,\epsilon}$ and that it satisfies (\ref{existence1}). First, for large $j$, we may assume that 
\begin{equation}
F_{2,\epsilon}(u_j) \ge F_{2,\epsilon}(1) \iff \lambda_2(u_j) - \int_\Sigma u_j^{-\epsilon}\;d\sigma_g \ge \lambda_2(\D_g) - 1,
\end{equation}
from which we deduce
\begin{equation}
\int_\Sigma u_j^{-\epsilon}\;d\sigma_g \le C(g).
\end{equation}
Now, the function $t\mapsto t^{-\frac{\epsilon}{N-2}}$ is convex on $\mathbb{R}^+$, and so is the functional
\begin{equation}\label{existence3}
\varphi \longmapsto \int_\Sigma \varphi^{-\frac{\epsilon}{N-2}}\;d\sigma_g 
\end{equation}
on $L^{n-1}_{\ge 0}(\Sigma)$. This functional is also strongly lower semicontinuous. Since $L^{n-1}_{\ge 0}(\Sigma)$ is strongly closed in $L^{n-1}(\Sigma)$, it follows that the functional in (\ref{existence3}) is weak lower semicontinuous on $L^{n-1}_{\ge 0}(\Sigma)$; see Chapter 3 in \cite{Brezis}. Therefore, 
\begin{equation}
\begin{split}
\int_\Sigma u_\epsilon^{-\epsilon}\;d\sigma_g =\int_\Sigma v^{-\frac{\epsilon}{N-2}}\;d\sigma_g &\le \liminf \int_\Sigma v_j^{-\frac{\epsilon}{N-2}}\;d\sigma_g = \liminf \int_\Sigma u_j^{-\epsilon}\;d\sigma_g\\& \le C(g).
\end{split}
\end{equation}
Hence, $u_\epsilon\in \D_\epsilon$. 

It remains to show that $u_\epsilon$ is maximal. By Lemma \ref{ConvergenceEigenvalues} below, up to the extraction of a subsequence, we have $\lim_{j\to \infty} \lambda_2(u_j) = \lambda_2(u_\epsilon)$. Therefore, using (\ref{existence2}), we get
\begin{equation}\label{existence4}
\limsup_{j\to\infty} \lambda_2(u_j) = \lambda_2(u_\epsilon) \le \lambda_2(u_\epsilon)\left(\int_\Sigma u_\epsilon^N\;d\sigma\right)^{\frac{N-2}{N}}.
\end{equation}
Moreover, using the weak lower semicontinuity of (\ref{existence3}),
\begin{equation}\label{existence5}
\left(\int_\Sigma u_\epsilon^{-\epsilon}\;d\sigma_g\right)\left(\int_\Sigma u_\epsilon^N\;d\sigma_g\right)^{\frac{\epsilon}{N}}\le \liminf_{j\to\infty} \int_\Sigma u_j^{-\epsilon}\;d\sigma_g.
\end{equation}
Putting (\ref{existence4}) and (\ref{existence5}) together, we deduce $\alpha\ge F_{2,\epsilon} (u_\epsilon)\ge \limsup_{j\to \infty} F_{2,\epsilon}(u_j) = \alpha$, which shows that $u_\epsilon$ is a maximizer. Finally, by scale invariance, we can scale to get $\int_\Sigma u_\epsilon^N\;d\sigma_g = 1$.  
\end{proof}

\begin{lemma}\label{RemainsBounded}
In the notation of the proof of Proposition \ref{existence}, we have
\begin{equation}
\lambda_1(u_j) \ge - C>-\infty,
\end{equation}
where $C$ is independent of $j$. 
\end{lemma}

\begin{proof}
Assume on the contrary that along some subsequence, still denoted by $\{u_j\}_{j=1}^\infty$, we have $\lim_{j\to \infty} \lambda_1(u_j) = -\infty$. Since each $u_j^{-1}(0)$ has zero (boundary) measure, by Proposition \ref{ExistenceSimplicity}, there is an associated sequence $\{f_j\}_{j=1}^\infty$ of first generalized eigenfunctions with $f_j> 0 $ a.e. on $\Sigma$ and $\int_\Sigma f_j^2u_j^{N-2}\;d\sigma_g = 1$. Since each of these satisfies
\begin{equation}\label{RB2}
\int_M |\nabla_g f_j|^2\;dv_g + 2c_nh_g \int_\Sigma f_j^2\;d\sigma  = \lambda_1(u_j)
\end{equation}
and $\lambda_1(u_j)\to -\infty$, we conclude that $\lim_{j\to\infty}\int_\Sigma f_j^2\;d\sigma_g = \infty \;\;(h_g<0)$.

Consider the new sequence $\tilde f_j = f_j\|f_j\|_{L^2(\Sigma)}^{-1}$. Due to linearity, each of these satisfies
\begin{equation}\label{RB5}
\int_M |\nabla_g \tilde f_j|^2\;dv_g + 2c_nh_g\int_\Sigma \tilde f_j^2\;d\sigma_g = \lambda_1(u_j)\int_\Sigma \tilde f_j^2u_j^{N-2}\;d\sigma_g.
\end{equation}
Therefore, $\int_M |\nabla_g\tilde f_j|^2\;dv_g \le 2c_n(-h_g)$, and we conclude that $\{\tilde f_j\}_{j=1}^\infty$ is bounded in $W^{1,2}(M)$ by (\ref{FriedrichsInequality}). Standard compactness then yields a function $f\in W^{1,2}(M)$ such that $\tilde f_j \rightharpoonup f \text{ in }W^{1,2}(M)$, $\tilde f_j \to f \text{ in }L^2(M)$, $\tilde f_j \to f \text{ in }L^{q}(\Sigma)$ for $1<q<N$. Since $\int_\Sigma \tilde f^2_j\;d\sigma_g = 1$ for all $j$, strong convergence gives $\int_\Sigma f^2\;d\sigma_g = 1$. On the other hand,
\begin{equation}\label{RB9}
\lim_{j\to \infty} \int_\Sigma \tilde f_j^2 u_j^{N-2}\;d\sigma_g = \lim_{j\to \infty}\frac{1}{\|f_j\|^2_{L^2(\Sigma)}} = 0,
\end{equation}
which implies
\[
0\le \lim_{j\to\infty}\int_\Sigma \tilde f_j u_j^{N-2}\;d\sigma_g \le \lim_{j\to \infty} \left(\int_\Sigma \tilde f_j^2u_j^{N-2}\;d\sigma_g\right)^{\frac{1}{2}}\left(\int_\Sigma u_j^{N-2}\;d\sigma_g\right)^{\frac{1}{2}} = 0.
\]
Furthermore, 
\begin{equation}
\begin{split}
\left|\int_\Sigma \tilde f_ju_j^{N-2} - fu_\epsilon^{N-2}\;d\sigma_g\right| &\le \left|\int_\Sigma (\tilde f_j - f)u_j^{N-2}\;d\sigma_g\right|\\&\hspace{.15in}+\left|\int_\Sigma f(u_j^{N-2}-u_\epsilon^{N-2})\;d\sigma_g\right|,
\end{split}
\end{equation}
where the first term goes to zero by H\"older's inequality, while the second term goes to zero by weak convergence in $L^{n-1}(\Sigma)$. This gives $\int_\Sigma fu_\epsilon^{N-2}\;d\sigma =0$, thus $f= 0$ on $\Sigma$ since $u_\epsilon>0$ a.e. on $\Sigma$, contradicting $\int_\Sigma f^2\;d\sigma_g = 1$. Hence, $\{\lambda_1(u_j)\}_{j=1}^\infty$ remains bounded. 
\end{proof}

\begin{lemma}\label{ConvergenceEigenvalues}
In the notation of the proof of Proposition \ref{existence}, we show that, after possibly extracting a subsequence, 
\begin{equation}\label{CE1}
\lim_{j\to\infty} \lambda_1(u_j) = \lambda_1(u_\epsilon) \text{ and }\lim_{j\to \infty}\lambda_2(u_j) = \lambda_2(u_\epsilon).
\end{equation}
\end{lemma}
\begin{proof}
By Lemma \ref{RemainsBounded}, up to the extraction of a subsequence, both limits $\lambda_1=\lim_{j\to\infty} \lambda_1(u_j)$ and $\lambda_2=\lim_{j\to \infty} \lambda_2(u_j)$ exist. Our goal is to show that $\lambda_1=\lambda_1(u_\epsilon)$ and $\lambda_2=\lambda_2(u_\epsilon)$. 
We show first that both $\lambda_1$ and $\lambda_2$ are generalized eigenvalues. 

By Proposition \ref{ExistenceSimplicity}, for each $u_j$, there are corresponding sequences of first and second generalized eigenfunctions $\{f_j\}_{j=1}^\infty$ and $\{\phi_j\}_{j=1}^\infty$ satisfying (\ref{EigenfunctionProp}). By Proposition \ref{Boundedness-SeqConfFact}, both sequences are bounded in $W^{1,2}(M)$. Moreover, these are bounded in $L^\infty(\Sigma)$ by Remark \ref{RMoserBound} after Proposition \ref{Estimates}. Let $f\ge 0$ and $\phi$ be the corresponding limit functions. Since each $f_j$ is a first generalized eigenfunction, we have 
\begin{equation}\label{CE2}
\int_M \langle\nabla_g f_j,\nabla_g\psi \rangle\;dv_g + 2c_n\int_\Sigma h_g f_j\psi\;d\sigma_g = \lambda_1(u_j) \int_\Sigma f_j \psi u_j^{N-2}\;d\sigma_g
\end{equation}
 for any test function $\psi\in C^\infty(\overline{M})$. By weak convergence in $W^{1,2}(M)$ and strong convergence in $L^2(\Sigma)$, the left hand side of (\ref{CE2}) goes to 
\begin{equation}\label{CE3}
\int_M \langle \nabla_g f, \nabla_g \psi \rangle\;dv_g + 2c_n\int_\Sigma h_g f\psi\;d\sigma_g.
\end{equation}
On the other hand, 
\begin{equation}
\begin{split}
\left|\int_\Sigma (f_j\psi u_j^{N-2} - f \psi u_\epsilon^{N-2}) \;d\sigma_g\right| \le &C\left|\int_\Sigma(f_j-f)u_j^{N-2}\;d\sigma_g\right|\\ &+ C\left|\int_\Sigma f(u_j^{N-2}-u_\epsilon^{N-2})\;d\sigma_g\right|,
\end{split}
\end{equation}
where $C=\|\psi\|_{L^\infty(\Sigma)}$. As in the proof of Lemma \ref{RemainsBounded}, the first term goes to zero by H\"older's inequality, while the second term goes to zero by weak convergence of the $u_j$'s in $L^{n-1}(\Sigma)$. Therefore, the right hand side of (\ref{CE2}) goes to
\begin{equation}\label{CE4}
\lambda_1\int_\Sigma f\psi u_\epsilon^{N-2}\;d\sigma_g.
\end{equation}
Putting (\ref{CE3}) and (\ref{CE4}) together, we conclude that $f$ is a weak solution of the system
\begin{equation}
\begin{cases}
\Delta_g f &= 0\; (\text{in } M);\\
B_g(f) &= \lambda_1 f u_\epsilon^{N-2}\; (\text{on } \Sigma). 
\end{cases}
\end{equation}
Hence, $\lambda_1$ is a generalized eigenvalue. A similar argument shows that $\lambda_2$ is a generalized eigenvalue as well. 

We claim that 
\begin{equation}\label{CE5}
\int_\Sigma f^2u_\epsilon^{N-2}\;d\sigma_g = \int_\Sigma \phi^2 u_\epsilon^{N-2}\;d\sigma_g = 1; \int_\Sigma f\phi u_\epsilon^{N-2}\;d\sigma_g = 0.
\end{equation}
We provide the arguments for the last equality as the details for all of them are similar. We estimate as follows:
\begin{equation}\label{CE6}
\begin{split}
\left|\int_\Sigma f\phi u_\epsilon^{N-2}\;d\sigma_g\right| = &\left| \int_\Sigma (f\phi u_\epsilon^{N-2} - f_j\phi_ju_j^{N-2})\;d\sigma_g\right| \\ \le &\underbrace{\left|\int_\Sigma (f\phi - f_j\phi_j)u_j^{N-2}\;d\sigma_g\right|}_{=A_j} \\ & +\left|\int_\Sigma f\phi(u_\epsilon^{N-2} - u_j^{N-2})\;d\sigma_g\right|.
\end{split}
\end{equation}
Since $|f\phi- f_j\phi_j| \le |\phi||f-f_j| + |f_j||\phi - \phi_j| \le C(|f-f_j| + |\phi - \phi_j|)$ by the uniform bound in $L^\infty(\Sigma)$, $f_j\phi_j \to f\phi$ in $L^{\frac{N}{2}}(\Sigma)$, which gives that $A_j$ goes to zero by H\"older's inequality. The second term in (\ref{CE6}) goes to zero by weak convergence of the $u_j$'s in $L^{n-1}(\Sigma)$.
 
Let us proceed now with the proof of $\lambda_1 = \lambda_1(u)$. Let $f_1$ be a first generalized eigenfunction, which we can assume to be positive a.e. on $\Sigma$ by Proposition \ref{ExistenceSimplicity}. From the variational formulation of the boundary eigenvalue problem, we have
\begin{equation}\label{CE7}
(\lambda_1(u_\epsilon) - \lambda_1)\int_\Sigma f_1fu_\epsilon^{N-2}\;d\sigma_g = 0.
\end{equation}
Since $f_1f\ge 0$ on $\Sigma$ and $u_\epsilon^{-1}(0)$ has zero measure, if $\lambda_1 \not = \lambda_1(u_\epsilon)$, then $f_1f=0$ a.e. on $\Sigma$. However, by Proposition \ref{ExistenceSimplicity}, $f_1{_{|_\Sigma}}$ can only vanish on a set of zero measure, thus $f=0$ on $\Sigma$. This contradicts (\ref{CE5}), hence $\lambda_1=\lambda_1(u_\epsilon)$ (and $f=f_1$ on $\Sigma$).

To prove that $\lambda_2=\lambda_2(u_\epsilon)$, we use Lemma \ref{Char2}. From the orthogonality condition in (\ref{CE5}), we deduce
\begin{equation}
\lambda_2(u_\epsilon) \le \lambda_2.
\end{equation}
To obtain the opposite inequality, let $\phi_\epsilon$ be an eigenfunction associated with $\lambda_2(u_\epsilon)$ satisfying 
\begin{equation}\label{CE8}
\int_\Sigma f^2u_\epsilon^{N-2}\;d\sigma_g = \int_\Sigma \phi_\epsilon^2 u_\epsilon^{N-2}\;d\sigma_g = 1; \int_\Sigma f\phi_\epsilon u_\epsilon^{N-2}\;d\sigma_g = 0.
\end{equation}
We apply Lemma \ref{Char2} to each element of the sequence $\{u_j\}_{j=1}^\infty$ using as a test function the projection of $\phi_\epsilon$ onto $E_1(u_j)$, that is, 
\begin{equation}
\begin{split}
\lambda_2(u_j) & \le \R_g^{u_j} \Bigg{(}\phi_\epsilon - \underbrace{\left(\int_\Sigma \phi_\epsilon f_ju_j^{N-2}\;d\sigma_g\right)}_{=c_j} \cdot f_j\Bigg{)} \\  &= \frac{\lambda_2(u_\epsilon) - c_j^2\lambda_1(u_j)}{\int_\Sigma \phi_\epsilon^2u_j^{N-2}\;d\sigma_g -c_j^2}.
\end{split}
\end{equation}
Since $c_j\to 0$ and $\int_\Sigma \phi_\epsilon^2u_j^{N-2}\;d\sigma_g \to 1$ as $j\to\infty$, we conclude
\begin{equation}
\lambda_2 = \lim_{j\to \infty}\lambda_2(u_j) \le \lambda_2(u_\epsilon),
\end{equation}
and the proof is complete. 
\end{proof}

In the next proposition, the Euler--Lagrange equation satisfied by each maximal function $u_\epsilon\in \D_\epsilon$ for $F_{2,\epsilon}$ is derived. The proof follows from a Hahn--Banach separation argument, which is now standard in eigenvalue optimization. For the proof in this specific context, see Proposition 4.3 in \cite{Gursky}. After replacing the interior Sobolev spaces by the corresponding boundary spaces, the argument is the same and thus we have decided to omit it here.

\begin{proposition}\label{EulerEq}
Let $u_\epsilon\in\D_\epsilon$ be a maximal function provided by Proposition \ref{existence}. Then there is a set of generalized eigenfunctions $\{\phi_{i,\epsilon}\}_{i=1}^{k_\epsilon}\subset E_2(u_\epsilon)$, normalized by 
\begin{equation}
\int_\Sigma \phi_{i,\epsilon}^2u_\epsilon^{N-2}\;d\sigma_g = 1,
\end{equation}
and a set of real numbers $c_{1,\epsilon},\cdots,c_{k_\epsilon,\epsilon}\ge 0$ with $\sum_{i=1}^{k_\epsilon}c_{i,\epsilon}=1$, such that 
\begin{equation}\label{EulerEq2}
\gamma_{1,\epsilon}u_\epsilon^N - u_{\epsilon}^{N-2}\sum_{i=1}^{k_\epsilon}c_{i,\epsilon}\phi_{i,\epsilon}^2 - \gamma_{2,\epsilon}u_\epsilon^{-\epsilon} = 0.
\end{equation}
\end{proposition}

\section{Uniform Estimates in $\epsilon$}\label{UniformEstimates}
The goal of this section is to establish $\epsilon$-independent estimates. Take any sequence of positive numbers $\epsilon_j \to 0^+$. Notice that the sequence $\{k_{\epsilon_j}\}_{j=1}^\infty$ arising from Proposition \ref{EulerEq} is a sequence of positive integers satisfying
\begin{equation}\label{MultiplicityBound}
1\le k_{\epsilon_j} \le N([g])-1,
\end{equation}
and thus up to a subsequence we can assume that $k=k_{\epsilon_j}$ is constant. Indeed, from Carath\'eodory's theorem we know that $k_{\epsilon_j} \le \frac{N([g])(N([g])-1)}{2}+1$. However, up to an appropriate change of basis of $E_2(u_{\epsilon_j})$, we can assume without loss of generality that $k_{\epsilon_j} \le N([g])-1$. What is relevant for us is that $\{k_{\epsilon_j}\}_{j=1}^\infty$ is uniformly bounded.

\begin{proposition}\label{EigenvalEstimate}
Let $\{u_\epsilon\}_{\epsilon>0}$ be the family of maximal functions provided by Proposition \ref{existence}. There exists a constant $C_2\not = C_2(\epsilon)$ such that
\begin{equation}\label{bound-eigenvalues}
|\lambda_2(u_\epsilon)|\le |\lambda_2(\D_g) - 1|.
\end{equation}
Moreover, if along some subsequence $\epsilon_j\to 0$ we have that $u_{\epsilon_j}\to u$ pointwise a.e. on $\Sigma$ and $u^{-1}(0)$ has zero boundary measure, then there exists a constant $C_1 \not = C_1(\epsilon_j)$ such that 
\begin{equation}\label{bound-eigenvalues2}
 \lambda_1(u_{\epsilon_j}) > -C_1.
\end{equation}
\end{proposition}
\begin{proof}
We start with the upper bound on $\{|\lambda_2(u_\epsilon)|\}_{\epsilon>0}$. For each $\epsilon>0$, we have $F_\epsilon(u_\epsilon)\ge F_\epsilon(1)$ since $u_\epsilon$ maximizes $F_\epsilon$ over $\D_\epsilon$. Therefore,
\begin{equation}
0> \lambda_2(u_\epsilon) \ge \lambda_2(u_\epsilon) - \int_\Sigma u_\epsilon^{-\epsilon}\;d\sigma_g \ge \lambda_2(\D_g) - 1.
\end{equation}
This gives the upper estimate on $|\lambda_2(u_\epsilon)|$. 

Let us now prove (\ref{bound-eigenvalues2}). Let $\{f_j\}_{j=1}^\infty$ be a sequence of associated first eigenfunctions given by Proposition \ref{ExistenceSimplicity}, and normalize the sequence such that $\int_\Sigma f_j^2u_{\epsilon_j}^{N-2}\;d\sigma_g = 1$. Then 
\begin{equation}\label{bound-eigenvalues3}
0>\lambda_1(u_{\epsilon_j}) = \int_M |\nabla_gf_j|^2\;dv_g + 2c_nh_g \int_\Sigma f_j^2\;d\sigma_g,
\end{equation}
and so
\begin{equation}\label{bound-eigenvalues4}
\int_M |\nabla_gf_j|^2\;dv_g < 2c_n (-h_g)\int_\Sigma f_j^2\;d\sigma_g.
\end{equation}
We claim that $\{f_j\}_{j=1}^\infty$ is bounded in $W^{1,2}(M)$. Notice that, by Friedrichs's inequality (\ref{FriedrichsInequality}), it is enough to show boundedness in $L^2(\Sigma)$. 

Let us assume, on the contrary, that $\|f_j\|_{L^2(\Sigma)}\to \infty$ and set $\tilde f_j = f_j \|f_j\|_{L^2(\Sigma)}^{-1}$. From rescaling (\ref{bound-eigenvalues3}) we now obtain $\int_M|\nabla_g\tilde f_j|^2\;dv_g < 2c_n(-h_g)$, and so $\{\tilde f_j\}_{j=1}^\infty$ is bounded in $W^{1,2}(M)$ thanks to (\ref{FriedrichsInequality}). Denote by $\tilde f$ its limiting function obtained by standard compactness arguments. Since $\int_\Sigma \tilde f_j^2\;d\sigma_g = 1$, we obtain $\int_\Sigma \tilde f^2\;d\sigma_g = 1$. However, Fatou's Lemma implies
\begin{equation}
0 \le \int_\Sigma \tilde f^2 u^{N-2}\;d\sigma_g \le \liminf_{j\to \infty} \int_\Sigma \tilde f_j^2 u_{\epsilon_j}^{N-2}\;d\sigma_g \to 0.
\end{equation}
Since $u^{-1}(0)$ has zero (boundary) measure, we conclude that $\tilde f = 0$ a.e. on $\Sigma$, contradicting that it has unit $L^2(\Sigma)$-norm. Hence, $\|f_j\|_{L^2(\Sigma)}$ remains bounded and so $\{f_j\}_{j=1}^\infty$ is bounded in $W^{1,2}(M)$. Using (\ref{bound-eigenvalues3}) and the Sobolev trace embedding, $\lambda_1(u_{\epsilon_j})$ remains bounded. This finishes the proof. 
\end{proof}

\begin{proposition}\label{bound-eigenfunctions}
Let $\{u_\epsilon\}_{\epsilon>0}$ be the family of maximal functions provided by Proposition \ref{existence}. Any sequence $\{\phi_\epsilon\}_{\epsilon>0}$ of generalized second eigenfunctions, normalized such that $\int_\Sigma \phi^2_\epsilon u^{N-2}_\epsilon\;d\sigma_g = 1$, is uniformly bounded in $W^{1,2}(M)\cap L^\infty(\Sigma)$. In particular, the associated family of second generalized eigenfunctions $\phi_{i,\epsilon}$ (arising from Proposition \ref{EulerEq}) is uniformly bounded in $W^{1,2}(M)\cap L^\infty(\Sigma)$: for each $i\in\{1,\cdots,k\}$, there is a $C\not=C(\epsilon)$, such that 
\begin{equation}
\|\phi_{i,\epsilon}\|_{W^{1,2}(M)} + \|\phi_{i,\epsilon}\|_{L^\infty(\Sigma)} \le C.
\end{equation}
\end{proposition}

\begin{proof}
By boundedness of $\{\phi_\epsilon\}_{\epsilon>0}$ in $W^{1,2}(M)\cap L^{\infty}(\Sigma)$ we mean that for any sequence $\epsilon_j\to 0^+$, the associated sequence $\{\phi_{\epsilon_j}\}_{j=1}^\infty$ is bounded in $W^{1,2}(M)\cap L^\infty(\Sigma)$. Since $\int_\Sigma u_{\epsilon_j}^N\;d\sigma_g = 1$ and because $\{\lambda_2(u_{\epsilon_j})\}_{j=1}^\infty$ is bounded by Proposition \ref{EigenvalEstimate}, the boundedness in $W^{1,2}(M)$ follows from Proposition \ref{Boundedness-SeqConfFact}. The uniform $L^\infty$-bounds are now a consequence of Proposition \ref{Estimates}; see also remarks \ref{GenEigenBounds} and \ref{RMoserBound}.
\end{proof}

We are now in a position to provide a positive lower bound for $\{|\lambda_2(u_\epsilon)|\}_{\epsilon}$.

\begin{proposition}
Let $\{u_\epsilon\}_{\epsilon>0}$ be the family of maximal functions provided by Proposition \ref{existence}. Then there exists a constant $C_2 \not = C_2(\epsilon)$ such that 
\begin{equation}
|\lambda_2(u_\epsilon)|\ge C_2 >0.
\end{equation}
\end{proposition}

\begin{proof}
Let us suppose, on the contrary, that we have a subsequence of $\{u_{\epsilon_j}\}_{j=1}^\infty$, still denoted by $\{u_{\epsilon_j}\}_{j=1}^\infty$, such that $0>\lambda_2(u_{\epsilon_j})\to 0$, where $\epsilon_j\to 0^+$ as $j\to \infty$. Associated to it, there is a sequence of second generalized eigenfunctions $\{\phi_{\epsilon_j}\}_{j=1}^\infty$, which we normalize to have unit $L^2(\Sigma, u_{\epsilon_j}^{N-2}d\sigma_g)$-norm, such that
\begin{equation}\label{EigenvalEstimate1}
\lambda_2(u_{\epsilon_j}) = \R_g^{u_{\epsilon_j}}(\phi_{\epsilon_j}) = \int_M |\nabla_g\phi_{\epsilon_j}|^2\;dv_g + 2c_n\int_\Sigma h_g\phi_{\epsilon_j}^2\;d\sigma_g \to 0.
\end{equation}
Thanks to Proposition \ref{bound-eigenfunctions}, $\{\phi_{\epsilon_j}\}_{j=1}^\infty$ is uniformly bounded in $W^{1,2}(M)\cap L^\infty(\Sigma)$.

Let $\phi\in W^{1,2}(M)$ be the limiting function. After using the eigenvalue equation and taking limits appropriately, we obtain that $\phi_{|_\Sigma}$ is an eigenfunction of $\D_g$ with associated eigenvalue $0$. Our assumptions on $\text{Spec}(\D_g)$ then imply that $\phi_{|_\Sigma} = 0$. On the other hand, since $\{v_j\}_{j=1}^\infty:= \{u_{\epsilon_j}^{N-2}\}_{j=1}^\infty$ is bounded in $L^{n-1}(\Sigma)$, it converges weakly in $L^{n-1}(\Sigma)$. Denote the limiting function by $v$, and set $u:=v^{\frac{1}{N-2}}$. Then

\begin{equation}
\begin{split}
\left|\int_\Sigma \phi^2u^{N-2}\;d\sigma_g - 1\right| &= \left|\int_\Sigma \phi^2u^{N-2}\;d\sigma_g - \int_\Sigma \phi^2_{\epsilon_j}u^{N-2}_{\epsilon_j}\;d\sigma_g\right| \\ & \le \int_\Sigma |\phi^2 - \phi_{\epsilon_j}^2|u_{\epsilon_j}^{N-2}\;d\sigma_g \\ & \hspace{.15in}+\left|\int_\Sigma \phi^2 (v - v_{j})\;d\sigma_g\right|.
\end{split}
\end{equation}
Notice that, thanks to the $L^\infty(\Sigma)$-estimates, $|\phi^2 - \phi^2_{\epsilon_j}| \le C |\phi - \phi_{\epsilon_j}|$, thus the first term goes to zero by H\"older's inequality and strong convergence of $\{\phi_{\epsilon_j}\}$ in $L^{\frac{N}{2}}(\Sigma)$. The second term goes to zero by weak convergence. Hence, $\int_\Sigma \phi^2u^{N-2}d\sigma_g = 1$, which contradicts $\phi_{|_{\Sigma}} = 0$. This finishes the proof.  
\end{proof}

\begin{corollary}\label{GammaBound}
Let $\gamma_{1,\epsilon}$ and $\gamma_{2,\epsilon}$ be the sequence provided by Proposition \ref{EulerEq} (see also Proposition \ref{existence}). Then
\begin{equation}
1\le \gamma_{1,\epsilon} \le 1+C\epsilon,
\end{equation}
and
\begin{equation}
C^{-1}\epsilon\le \gamma_{2,\epsilon}\le C\epsilon,
\end{equation}
where $C$ is a generic constant independent of $\epsilon$.
\end{corollary}

\begin{proposition}\label{u-Linfinity}
There is a constant $C$, independent of $\epsilon$, such that
\begin{equation}
\|u_\epsilon\|_{L^\infty(\Sigma)}\le C.
\end{equation}
\end{proposition}
\begin{proof}
By Proposition \ref{EulerEq}, we know that 
\begin{equation}
\gamma_{1,\epsilon}u_\epsilon^N = u_\epsilon^{N-2}\sum_{i=1}^kc_{i,\epsilon}\phi_{i,\epsilon}^2 + \gamma_{2,\epsilon}u_\epsilon^{-\epsilon}.
\end{equation}
For any point $x\in \Sigma$ at which $u_\epsilon(x)\ge 1$, we obtain
\begin{equation}
\begin{split}
u_\epsilon(x)^N &\le \gamma_{1,\epsilon}u_{\epsilon}(x)^N =   u_\epsilon(x)^{N-2}\sum_{i=1}^k c_{i,\epsilon}\phi_{i,\epsilon}(x)^2 +\gamma_{2,\epsilon}u_\epsilon(x)^{-\epsilon} \\ &\le u_\epsilon(x)^{N-2}\sum_{i=1}^kc_{i,\epsilon}\phi_{i,\epsilon}(x)^2 + \gamma_{2,\epsilon} \le C(u_\epsilon(x)^{N-2}+\epsilon), 
\end{split}
\end{equation}
where the last inequality follows from Propositions \ref{GammaBound} and \ref{bound-eigenfunctions}. This implies $u_\epsilon(x)^2\le C$, and thus the proof is completed. 
\end{proof}

\begin{proposition}\label{C0W1p-bounds}
For any $\alpha\in (0,1)$ and $p\ge2$, there is a constant $C=C(\alpha,p,g)$ such that 
\begin{equation}
\|\phi_{i,\epsilon}\|_{C^{0,\alpha}(\Sigma)} + \|\phi_{i,\epsilon}\|_{W^{1,p}(\Sigma)}\le C
\end{equation}
for each $i$.
\end{proposition}

\begin{proof}
Since $\mathcal{D}_g$ is an elliptic pseudodifferential operator of order $1$, elliptic regularity \cite{TaylorPDO} implies
\begin{equation}
\begin{split}
\|\phi_{i,\epsilon}\|_{W^{1,p}(\Sigma)} &\le C(\|\lambda_2(u_\epsilon)\phi_{i,\epsilon}u_\epsilon^{N-2}\|_{L^p(\Sigma)} + \|\phi_{i,\epsilon}\|_{L^p(\Sigma)}) \le \tilde C,
\end{split}
\end{equation}
where the last inequality follows from Propositions \ref{EigenvalEstimate}, \ref{bound-eigenfunctions} and \ref{u-Linfinity}, and $\tilde C$ is a generic constant independent of $\epsilon$. The preceding estimate holds for every finite $p \ge 2$. Given
$\alpha \in (0,1)$, choose $q>\frac{n-1}{1-\alpha}$. Morrey's embedding $W^{1,q}(\Sigma)\hookrightarrow C^{0,\alpha}(\Sigma)$ then gives the required uniform Hölder estimate.
\end{proof}

The following $\epsilon$-dependent lower bound can be obtained for the sequence of extremals $\{u_\epsilon\}_{\epsilon>0}$.

\begin{proposition}\label{u-lowerbound}
There is a constant $C$ such that 
\begin{equation}
\essinf u_\epsilon \ge (C\epsilon)^\frac{1}{\epsilon+N}.
\end{equation}
\end{proposition}

\begin{proof}
Pick a point $x\in \Sigma$ such that $u_\epsilon(x)>0$. From (\ref{EulerEq2}), at $x$ we have
\begin{equation}
\begin{split}
\gamma_{1,\epsilon} u_\epsilon^N(x) &= u_\epsilon^{N-2}(x) \sum_{i=1}^k c_{i,\epsilon}\phi_{i,\epsilon}^2(x) + \gamma_{2,\epsilon}u_\epsilon^{-\epsilon}(x) \ge \gamma_{2,\epsilon} u_\epsilon^{-\epsilon}(x),
\end{split}
\end{equation}
from which we obtain
\begin{equation}
\gamma_{1,\epsilon} u_\epsilon^{N+\epsilon}(x) \ge \gamma_{2,\epsilon}.
\end{equation}
The result follows from Corollary \ref{GammaBound}.
\end{proof}

Our next goal is to prove uniform estimates for the family $\{u_\epsilon\}_{\epsilon>0}$ of maximal generalized conformal factors in $C^{0,\mu}(\Sigma)$, for some $\mu\in(0,1)$. To this end, let us  rearrange (\ref{EulerEq2}) as follows:
\begin{equation}\label{EulerEq3}
\gamma_{1,\epsilon} u_\epsilon^2 - \gamma_{2,\epsilon}u_\epsilon^{-\epsilon - N + 2}= \sum_{i=1}^k c_{i,\epsilon}\phi_{i,\epsilon}^2.
\end{equation}
If we define $f_\epsilon:\mathbb{R}^+ \to \mathbb{R}$ by
\begin{equation}
f_\epsilon(t) = \gamma_{1,\epsilon}t^2 - \gamma_{2,\epsilon} t^{-\epsilon-N+2},
\end{equation}
and $\Phi_\epsilon: \Sigma\to \mathbb{R}$ by 
\begin{equation}\label{PhiDef}
\Phi_\epsilon(x) = \sum_{i=1}^kc_{i,\epsilon}\phi_{i,\epsilon}^2(x),
\end{equation}
then (\ref{EulerEq3}) reads as $f_\epsilon(u_\epsilon) = \Phi_\epsilon$.

The function $f_\epsilon$ is strictly increasing on $\mathbb{R}^+$. Indeed, direct computation shows
\begin{equation}
f_\epsilon'(t) = 2\gamma_{1,\epsilon}t + \gamma_{2,\epsilon}(\epsilon + N -2) t^{-\epsilon - N +1}\ge \min_{\mathbb{R}^+} f_\epsilon' >0. 
\end{equation}
Therefore, its inverse $f_\epsilon^{-1}: \mathbb{R} \to \mathbb{R}^+$ exists and it is also strictly increasing. From the equality $f_\epsilon(u_\epsilon) = \Phi_\epsilon$, we then deduce $u_\epsilon = f_\epsilon^{-1} \circ \Phi_\epsilon$. Using
\begin{equation}
(f_\epsilon^{-1})'(t) = \frac{1}{f_\epsilon'(f_\epsilon^{-1}(t))} \le \frac{1}{\min_{\mathbb{R}^+} f_\epsilon'},
\end{equation}
we observe that $(f_\epsilon^{-1})'(t)$ remains bounded, thus $f_\epsilon^{-1}$ is Lipschitz for each fixed $\epsilon>0$. As a consequence, we have the desired result:

\begin{proposition}\label{u-HolderBound}
The family $\{u_\epsilon\}_{\epsilon>0}$ is uniformly bounded in  both $W^{1,p}(\Sigma)$ and  $C^{0,\mu}(\Sigma)$, for $2 \le p < \infty$ and for some $\mu\in (0,1)$.
\end{proposition}

\begin{proof}
From the preamble, we have that $u_\epsilon = f_\epsilon^{-1}\circ\Phi_\epsilon$ and that $f_\epsilon^{-1}$ is Lipschitz. Moreover, from Propositions \ref{u-Linfinity} and \ref{C0W1p-bounds}, it follows that $u_\epsilon \in L^p(\Sigma)$ and $\Phi_\epsilon\in W^{1,p}(\Sigma)$. Hence, $u_\epsilon \in W^{1,p}(\Sigma)$ and $|\nabla_\Sigma u_\epsilon| = |(f^{-1})'(\Phi_\epsilon)|\cdot |\nabla_\Sigma \Phi_\epsilon|$ almost everywhere in $\Sigma$, where $\nabla_\Sigma$ denotes the induced gradient on $\Sigma$ (see Proposition 2.5 in \cite{Hebey}). 

It remains to show that the $W^{1,p}(\Sigma)$-bounds are uniform. Notice that by Proposition \ref{u-Linfinity}, we have already uniform estimates for the $L^p$-norm of $u_\epsilon$. To show uniform estimates for the $L^p$-norm of $|\nabla_\Sigma u_\epsilon|$ we proceed as follows. First, from (\ref{PhiDef}) we obtain that
\begin{equation}
|\nabla_\Sigma \Phi_\epsilon| = 2\left|\sum_{i=1}^kc_{i,\epsilon}\phi_{i,\epsilon}\nabla_\Sigma \phi_{i,\epsilon}\right|
\end{equation}
holds almost everywhere on $\Sigma$. Then, by the Cauchy-Schwarz inequality, 
\begin{equation}\
|\nabla_\Sigma \Phi_\epsilon| \le 2 \Phi_\epsilon^{\frac{1}{2}} \left(\sum_{i=1}^k c_{i,\epsilon}|\nabla_\Sigma \phi_{i,\epsilon}|^2\right)^{\frac{1}{2}}.
\end{equation}
Going back to (\ref{EulerEq3}), we deduce that 
\begin{equation}\label{u-HolderBound-1}
|\nabla_\Sigma \Phi_\epsilon| \le Cu_\epsilon\left(\sum_{i=1}^kc_{i,\epsilon}|\nabla_\Sigma \phi_{i,\epsilon}|^2\right)^{\frac12}.
\end{equation}
Using that $f_\epsilon'(t) \ge 2t$ on $\mathbb{R}^+$ and the convexity of $\mathbb{R}^+\ni x\mapsto x^{\frac{p}{2}}$ for $p\ge 2$, it then follows from (\ref{u-HolderBound-1}) and from Jensen's inequality that
\begin{equation}
|\nabla_\Sigma u_\epsilon|^p = \frac{|\nabla_\Sigma \Phi_\epsilon|^p}{|f'_\epsilon(u_\epsilon)|^p} \le C \sum_{i=1}^k |\nabla_\Sigma \phi_{i,\epsilon}|^p
\end{equation}
almost everywhere on $\Sigma$. Therefore, for $p\ge 2$, we have \[\|\nabla_\Sigma u_\epsilon\|_{L^p(\Sigma)}\le C\sum_{i=1}^k\|\nabla_\Sigma \phi_{i,\epsilon}\|_{L^p(\Sigma)} \le C \] due to Proposition \ref{C0W1p-bounds}. Together with Proposition \ref{u-Linfinity}, this shows the uniform boundedness of $\{u_\epsilon\}_{\epsilon>0}$ in $W^{1,p}(\Sigma)$. The uniform estimates in the $C^{0,\mu}$-norm now follow by Morrey’s embedding, after choosing the exponent sufficiently large. 
\end{proof}

\begin{proposition}\label{KeyEstimate}
There exists a constant $C$, independent of $\epsilon$, such that
\begin{equation}\label{KeyEstimate2}
\epsilon \int_\Sigma u_\epsilon^{-\epsilon - N}\;d\sigma_g \le C.
\end{equation}
Moreover, for any $\eta\in C^\infty(\Sigma)$, the following holds as $\epsilon \to 0^+$:
\begin{equation}\label{KeyEstimate3}
\int_\Sigma \eta \left\{ \gamma_{1,\epsilon}u^2_\epsilon - \sum_{i=1}^k c_{i,\epsilon}\phi_{i,\epsilon}^2 \right\}\;d\sigma_g = O(\epsilon^{\frac{2}{N+\epsilon}})\|\eta\|_{L^\infty(\Sigma)}.
\end{equation}
\end{proposition}
\begin{proof}
Multiply (\ref{EulerEq2}) by $u_\epsilon^{-N}$ and integrate to get
\begin{equation}
\gamma_{2,\epsilon}\int_{\Sigma} u_\epsilon^{-\epsilon-N}\;d\sigma_g = \gamma_{1,\epsilon} - \int_{\Sigma}u_\epsilon^{-2}\sum_{i=1}^kc_{i,\epsilon}\phi_{i,\epsilon}^2\;d\sigma_g\le \gamma_{1,\epsilon}.
\end{equation}
Estimate (\ref{KeyEstimate2}) now follows from Corollary \ref{GammaBound}. For the proof of (\ref{KeyEstimate3}), let $\eta\in C^\infty(\Sigma)$ be arbitrary, and multiply (\ref{EulerEq2}) by $\eta u_\epsilon^{2-N}$ to get
\begin{equation}\label{KeyEstimate4}
\int_\Sigma \eta\left\{\gamma_{1,\epsilon}u_\epsilon^2 - \sum_{i=1}^kc_{i,\epsilon}\phi_{i,\epsilon}^2 \right\}\;d\sigma_g = \underbrace{\gamma_{2,\epsilon}\int_\Sigma \eta u_\epsilon^{2-N-\epsilon}}_{:=A_\epsilon}\;d\sigma_g.
\end{equation}
Applying H\"older's inequality on the right-hand side of (\ref{KeyEstimate4}) gives
\begin{equation}
\begin{split}
A_\epsilon & \le C\|\eta\|_{L^\infty(\Sigma)}\cdot\epsilon \left(\int_\Sigma u_\epsilon^{-\epsilon-N}\;d\sigma_g\right)^{\frac{2-N-\epsilon}{-\epsilon - N}} \\ &=C\|\eta\|_{L^\infty(\Sigma)}\cdot \epsilon^{\frac{2}{N+\epsilon}}\left(\epsilon\int_\Sigma u_\epsilon^{-\epsilon-N}\;d\sigma_g\right)^{\frac{N+\epsilon-2}{N+\epsilon}},
\end{split}
\end{equation}
where to get the inequality we have used Corollary \ref{GammaBound}.
Estimate (\ref{KeyEstimate2}) then implies estimate (\ref{KeyEstimate3}). This completes the proof. 
\end{proof}

\section{Taking the limit $\epsilon\to0^+$: Proof of Theorem \ref{MainTheorem}}\label{Tlimit}

Let $\{u_\epsilon\}_{\epsilon>0}$ be the family of maximal functions for $F_{2,\epsilon}$ provided by Proposition \ref{existence}. Take any sequence $\epsilon_j>0$ such that $\epsilon_j \to 0^+$ as $j\to\infty$, and set $u_j:=u_{\epsilon_j}$. The associated sequence of generalized second eigenfunctions is denoted by $\{\phi_{i,\epsilon_j}\}_{j =1}^\infty$, where $i\in\{1,\cdots,k\}$ is fixed. We will keep the same notation for any subsequence.

By Proposition \ref{u-HolderBound} and the compact embedding of H\"older spaces, the sequence $\{u_j\}_{j\ge1}$ is relatively compact in $C^{0,\beta}(\Sigma)$ for some $0< \beta < \mu$. In particular, and up to the extraction of a subsequence, there exists a $u \in C^{0,\beta}(\Sigma)$ such that 
\begin{equation}\label{l1}
u_j \to u \text{ in }C^{0,\beta}(\Sigma).
\end{equation}
Since $\|u_j\|_{L^N(\Sigma)} = 1$ and $u_j\ge 0$ on $\Sigma$ for all $j\ge 1$, we conclude that $u$ is non-negative on $\Sigma$ and 
\begin{equation}\label{l2}
\int_\Sigma u^N\;d\sigma_g = 1.
\end{equation}
Another consequence of Proposition \ref{u-HolderBound} is that $u\in W^{1,2}(\Sigma)$. 

We now deal with the associated sequence of eigenfunctions $\{\phi_{i,\epsilon_j}\}_{j=1}^\infty$. As these are generalized second eigenfunctions, they satisfy 
\begin{equation}\label{l3}
\begin{cases}
\Delta_g \phi_{i,\epsilon_j} & = 0 \; (\text{in } M); \\
B_g(\phi_{i,\epsilon_j}) & = \lambda_2(u_j) \phi_{i,\epsilon_j}u_j^{N-2} \;(\text{on }\Sigma).
\end{cases}
\end{equation}
with the normalization
\begin{equation}\label{l4}
\int_\Sigma \phi_{i,\epsilon_j}^2u_j^{N-2}\;d\sigma_g = 1.
\end{equation}
By Propositions \ref{bound-eigenfunctions} and \ref{C0W1p-bounds}, up to the extraction of a subsequence, for each $i\in \{1,\cdots,k\}$, there exists $\phi_i\in C^{0,\beta}(\Sigma)\cap W^{1,2}(M)$ (with a possibly smaller $\beta$) such that 
\begin{equation}\label{l5}
\begin{split}
\phi_{i,\epsilon_j} &\to \phi_i \text{ in }C^{0,\beta}(\Sigma),\\
\phi_{i,\epsilon_j} &\rightharpoonup \phi_i \text{ in } W^{1,2}(M).
\end{split}
\end{equation}
From (\ref{l4}), we deduce that
\begin{equation}\label{l6}
\int_\Sigma \phi_i^2u^{N-2}d\sigma_g = 1.
\end{equation}
On the other hand, it follows from Proposition \ref{EigenvalEstimate} that $\lim_{j\to\infty} \lambda_2(u_j)$ exists after possibly extracting a subsequence. We denote this limit by $\lambda_2$.

We are now in a position to take limits. The variational characterization of (\ref{l3}) gives
\begin{equation}\label{l7}
\int_M \langle \nabla_g \phi_{i,\epsilon_j}, \nabla_g \psi\rangle\;dv_g + 2c_n \int_\Sigma h_g\phi_{i,\epsilon_j}\psi \;d\sigma_g = \lambda_2(u_{\epsilon_j})\int_\Sigma \phi_{i,\epsilon_j}\psi u_j^{N-2}\;d\sigma_g
\end{equation}
for all $\psi\in W^{1,2}(M)$. As a consequence of (\ref{l5}), we obtain
\begin{equation}\label{l8}
\int_M\langle\nabla_g\phi_i,\nabla_g\psi\rangle\;dv_g + 2c_n\int_\Sigma h_g \phi_i\psi\;d\sigma_g = \lambda_2\int_\Sigma \phi_i\psi u^{N-2}\;d\sigma_g.
\end{equation}
Since $\psi \in W^{1,2}(M)$ is arbitrary, each $\phi_i$ satisfies
\begin{equation}\label{l9}
\begin{cases}
\Delta_g \phi_i & = 0\; (\text{in } M); \\
B_g(\phi_i) & = \lambda_2 \phi_iu^{N-2}\; (\text{on }\Sigma),
\end{cases}
\end{equation}
and thus elliptic regularity implies that $\phi_i\in C^{1,\beta}(\Sigma)\cap C^2(M^n)$. Furthermore, by (\ref{KeyEstimate3}) in Proposition \ref{KeyEstimate}, for any $\eta\in C^\infty(\Sigma)$, the following limit holds:
\begin{equation}\label{l10}
\lim_{j\to \infty} \int_\Sigma \eta\left\{\gamma_{1,\epsilon_j}u_j^2 - \sum_{i=1}^kc_{i,\epsilon_j}\phi_{i,\epsilon_j}^2\right\}\;d\sigma_g = \int_\Sigma \eta\left\{u^2 - \sum_{i=1}^kc_i\phi_i^2\right\}\;d\sigma_g = 0,
\end{equation}
where $c_i = \lim_{j\to\infty} c_{i,\epsilon_j}$ (after possibly extracting a subsequence). Notice that for the convergence of $\gamma_{1,\epsilon_j}$ we have used Corollary \ref{GammaBound}. Also observe that not all $c_i$ are zero since $\sum_{i=1}^kc_{i,\epsilon_j} = 1$, and so we may assume that $c_i>0$ for all $i$ after discarding those which are zero. From (\ref{l10}) we conclude that
\begin{equation}\label{l11}
\sum_{i=1}^kc_i\phi_i^2 = u^2 
\end{equation}
holds on $\Sigma$. By rescaling, we could redefine each $\phi_i$ by $\sqrt{c_i}\phi$ and get (\ref{MT2}) as claimed in Theorem \ref{MainTheorem}. From now on, we assume this has been done. 

There are still a few remaining aspects to complete the proof of Theorem \ref{MainTheorem}. We start with the following:

\begin{claim}\label{ZeroBoundaryMeasure}
$u^{-1}(0)$ has zero (boundary) Riemannian measure. In particular, $u\in L^N_{>0}(\Sigma)$.
\end{claim}

\begin{proof}
The result follows from Li's work \cite{Li}; see Theorem 1.1, part (b). Indeed, by (\ref{l11}), $u^{-1}(0) = \cap_{i=1}^k (\phi_i^{-1}(0)\cap \Sigma) \subseteq \phi_{i_o}^{-1}(0)\cap \Sigma$ for some $i_o\in \{1,\cdots, k\}$. Since $\phi_{i_o}$ solves (\ref{l9}), it solves $\Delta_g \phi_{i_o} = 0$ (in $M$), $\partial_{\nu_g}\phi_{i_o} = (\lambda_2u^{N-2} - 2c_nh_g)\phi_{i_o}$ (on $\Sigma$) after rearranging. Moreover, since $\phi_{i_o}\in C^{0,\beta}(\Sigma)\cap W^{1,2}(M)$ and $u\in C^{0,\beta}(\Sigma)$, the hypotheses of Theorem 1.1, part (b), in \cite{Li} are met, thus it follows that the Hausdorff dimension of $\phi_{i_o}^{-1}(0)\cap \Sigma$ is at most $n-2$. Hence, $u^{-1}(0)\subseteq \phi_{i_o}^{-1}(0)\cap \Sigma$ has zero (boundary) Riemannian measure. 
\end{proof}

Since the limiting function $u$ is positive a.e. on $\Sigma$, we have the boundedness of $\{\lambda_1(u_{j})\}_{j=1}^\infty$ thanks to Proposition \ref{EigenvalEstimate}. Set $\lambda_1 = \lim_{j\to\infty} \lambda_1(u_j)$, which is possible after passing to a further subsequence. 

\begin{claim}\label{LimitEigenvalues}
$\lambda_2$ coincides with $\lambda_2(u)$. Hence, $\phi_i\in E_2(u)$ for each $i=1,\cdots, k$.
\end{claim}

\begin{proof}
The main ideas here are as in the proof of Lemma \ref{ConvergenceEigenvalues}. Let $\{f_j\}_{j=1}^\infty$ be the sequence of first eigenfunctions associated to $\{u_j\}_{j=1}^\infty$. By Proposition \ref{ExistenceSimplicity}, we can assume that each $f_j$ is positive a.e. on $\Sigma$, and also we normalize such that $\int_\Sigma f_j^2u_j^{N-2}\;d\sigma_g = 1$. The boundedness of $\{f_j\}_{j=1}^\infty$ follows from Proposition \ref{Boundedness-SeqConfFact}. Denote by $f\ge 0$ its limiting function, and recall that $\lambda_1 = \lim_{j\to \infty}\lambda_1(u_j)$. Repeating the argument from the proof of Lemma \ref{ConvergenceEigenvalues}, the same equation (\ref{CE7}) can be derived. This yields $\lambda_1 = \lambda_1(u)$ and $f$ is a first eigenfunction with $f>0$ a.e. on $\Sigma$. 

Observe that $\lambda_1(u_j)<\lambda_2(u_j)$ by Proposition \ref{ExistenceSimplicity}, and so $\int_\Sigma f_j\phi_{i,j}u_j^{N-2}\;d\sigma_g = 0$ for all $j\in \mathbb{N}$ and for all $i\in \{1,\cdots, k\}$. Therefore, $\int_\Sigma f\phi_iu^{N-2}\;d\sigma_g = 0$ for every $i\in \{1,\cdots, k\}$. Due to Lemma \ref{Char2}, it then follows that $\lambda_2(u)\le \lambda_2$. To prove the opposite inequality, let $\psi_2$ be an eigenfunction associated to $\lambda_2(u)$, and set $c_j = \int_\Sigma \psi_2f_ju_j^{N-2}\;d\sigma_g$. Then $c_j \to \int_\Sigma \psi_2fu^{N-2}\;d\sigma_g = 0$ and so 
\begin{equation}
\lambda_2(u_j)  \le \R_g^{u_j} \left(\psi_2 - c_jf_j\right).
\end{equation}
Taking limits yields $\lambda_2 \le \lambda_2(u)$, as desired. This finishes the proof.

\end{proof}

\begin{claim}\label{u-maximal}
$u$ maximizes $\lambda_2$ over $L^N_{>0}(\Sigma)$, that is, $\lambda_2(u)$ is maximal. 
\end{claim}

\begin{proof}
Let $u$ be as constructed, and assume on the contrary that there exists a generalized conformal factor $w\in L^N_{>0}(\Sigma)$ with $\int_\Sigma w^N\;d\sigma_g = 1$ and with $\lambda_2(w) = \lambda_2(u) + \eta$ for some $\eta>0$. Set $w_\delta:=\sup\{w,\delta\}$ for $\delta>0$ so that $w_\delta\in \mathcal{D}_\epsilon$ for $\epsilon>0$. 

We will show first that $\delta>0$ can be chosen such that
\begin{equation}\label{u-Maximal-1}
\lambda_2(w_\delta)\left(\int_\Sigma w_\delta^N\;d\sigma_g\right)^{\frac{N-2}{N}} > \lambda_2(u)+\frac12 \eta.
\end{equation}
By Proposition \ref{NegativeEigen}, the number of generalized negative eigenvalues of $w_\delta$ and $w$ is still $N([g])$. In particular, both $\lambda_2(w)$ and $\lambda_2(w_\delta)$ are negative. Since $w_\delta\ge w$ and since it is enough to consider those $\phi\in W^{1,2}(M)$ for which 
\begin{equation}
\int_M|\nabla_g \phi|^2\;dv_g + 2c_nh_g\int_\Sigma \phi^2\;d\sigma_g<0, 
\end{equation}
we deduce that $\lambda_2(w_\delta)\ge \lambda_2(w)$. As for the volume, we obtain
\begin{equation}
\int_\Sigma w_\delta^N\;d\sigma_g \le \delta^N+1
\end{equation} 
using that $\text{Vol}(\Sigma,d\sigma_g) =1$ and $\int_\Sigma w^N\;d\sigma_g = 1$. Since $\frac{N-2}{N}<1$, we have that $(\delta^N+1)^{\frac{N-2}{N}}\le 1+ \frac{N-2}{N}\delta^N$ and so
\begin{equation}
\left(\int_\Sigma w_\delta^N\;d\sigma_g\right)^{\frac{N-2}{N}} \le 1 + \frac{N-2}{N}\delta^N. 
\end{equation}
Inequality (\ref{u-Maximal-1}) now follows after choosing $\delta>0$ small enough. 

Recall that $u_j=u_{\epsilon_j}$ and that $\epsilon_j\to 0$ as $j\to \infty$, where $u_j$ maximizes $F_{2,\epsilon_j}$ over $\mathcal{D}_{\epsilon_j}$. As a consequence, 
\begin{equation}\label{u-Maximal-2}
\begin{split}
F_{2,\epsilon_j}(u_{\epsilon_j}) &\ge F_{2,\epsilon_j}(w_\delta) \\ &\ge \lambda_2(u)+\frac12 \eta - \left(\int_\Sigma w_\delta^{-\epsilon_j}\;d\sigma_g\right)\left(\int_\Sigma w_\delta^N\;d\sigma_g\right)^{\frac{\epsilon_j}{N}}.
\end{split}
\end{equation}
By Claim \ref{LimitEigenvalues}, we know that $\lim_{j\to \infty}\lambda_2(u_j) = \lambda_2 = \lambda_2(u)$. On the other hand, thanks to Proposition \ref{KeyEstimate}, 
\begin{equation}
\begin{split}
\int_\Sigma u_j^{-\epsilon_j}\;d\sigma_g &\le \left(\int_\Sigma u_j^{-N-\epsilon_j}\;d\sigma_g\right)^{\frac{\epsilon_j}{N+\epsilon_j}}\text{Vol}(\Sigma)^{\frac{N}{N+\epsilon_j}} \le C^{\frac{\epsilon_j}{N+\epsilon_j}} (\epsilon_j^{-1})^{\frac{\epsilon_j}{N+\epsilon_j}} \to 1
\end{split}
\end{equation}
as $j\to \infty$. Moreover,
\begin{equation}
\begin{split}
1 = \int_\Sigma u_j^{\frac{-\epsilon_j}{2}}u_j^{\frac{\epsilon_j}{2}}\;d\sigma_g &\le \left(\int_\Sigma u_j^{-\epsilon_j}\;d\sigma_g\right)^{\frac12}\left(\int_\Sigma u_j^{\epsilon_j}\;d\sigma_g\right)^{\frac12} \\&\le C^{\epsilon_j} \left(\int_\Sigma u_j^{\epsilon_j}\;d\sigma_g\right)^{\frac12},
\end{split}
\end{equation}
where we have used Proposition \ref{u-HolderBound}. Therefore, $\lim_{j\to \infty}\int_\Sigma u_j^{-\epsilon_j}\;d\sigma_g = 1$. Taking limits in (\ref{u-Maximal-2}) and using the fact that $w_\delta\ge \delta>0$ then yields
\begin{equation}
\lambda_2(u) - 1 \ge \lambda_2(u) + \frac12 \eta -1.
\end{equation}
This is a contradiction. Hence, the proof is complete. 
\end{proof}

\begin{claim}
$u$ is smooth outside its zero set, i.e. $u\in C^\infty(\Sigma\setminus u^{-1}(0))$.
\end{claim}
\begin{proof}
Recall that $u\in C^{0,\beta}(\Sigma)$, which implies that $\Sigma\setminus u^{-1}(0)$ is open. In open subsets of $\Sigma\setminus u^{-1}(0)$, it follows from (\ref{l11}) that we have $u=\sqrt{\sum_{i=1}^k\phi_i^2}$. Therefore, we can differentiate using that $\phi_i\in C^{1,\beta}(\Sigma)$. A standard bootstrap argument using (\ref{l9}) finishes the proof of the claim. 
\end{proof}

\begin{claim}\label{HE}
The harmonic extension $\hat u$ of $u$ is smooth and positive in the interior $M$. Moreover, $\hat u\in W^{1,2}(M)$.
\end{claim}

\begin{proof}
By standard existence results (see Chapter 2 in \cite{Gilbarg}), given $u\in C(\Sigma)\cap W^{1,2}(\Sigma)$, there is a unique $\hat u\in C^2(M)\cap W^{1,2}(M)$ solution of 
\begin{equation}
\begin{cases}\label{HE1}
\Delta_g \hat u &= 0\; (\text{in } M);\\
\hat u &= u \;(\text{on }\Sigma).
\end{cases}
\end{equation}
Interior elliptic regularity then gives the smoothness of $\hat u$ in the interior of $M$. It follows from the classical maximum principle (Theorem 3.1 in \cite{Gilbarg}, for instance) that the maximum and minimum of $\hat u$ are attained on $\Sigma$. Since $u$ is nonnegative on $\Sigma$, $\hat u$ is nonnegative in $M$. Moreover, if there is an interior point $x_o\in M$ such that $\hat u(x_o) = 0$, then by connectedness we have that $\hat u \equiv 0$ on $M$, contradicting that $u>0$ a.e. on $\Sigma$. This finishes the proof. 
\end{proof}

\begin{claim}\label{HM}
Assume $k>1$ and set $\Sigma^*:=\Sigma\setminus u^{-1}(0)$. Then the map \newline$\Phi=(\frac{\phi_1}{\hat u},\cdots, \frac{\phi_k}{\hat u}): (M, g_{\hat u})\to (\mathbb{B}^k,g_E)$ satisfies $\Phi(\Sigma^*)\subset \mathbb{S}^{k-1}$ and it is a harmonic map with $\partial_{\nu_{g_{\hat u}}}\Phi \| \Phi$ on $\Sigma^*$. Here $g_{\hat u} = \hat u^{\frac{4}{n-2}}g$ is the conformal metric defined by the harmonic extension $\hat u$ of $u$, which agrees with $u^{\frac{4}{n-2}}g$ along $\Sigma$.
\end{claim}
\begin{proof}
Recall that $\sum_{i=1}^k\phi_i^2 = u^2$ on $\Sigma$. We set $\Psi = (\phi_1,\cdots, \phi_k)$.  Notice that, for any $v\in\mathbb{S}^{k-1}$, the function $g_{E}(\Psi,v) - \hat u$ is harmonic. On the other hand, $g_E(\Psi,v)\le |\Psi| = u = \hat u$ on $\Sigma$, and so the maximum principle yields $g_E(\Psi, v)\le \hat u$ on $\overline M$. Since $v\in \mathbb{S}^{k-1}$ is arbitrary, this yields $|\Psi|\le \hat u$ in $M$.

We now set $\hat \phi_i = \frac{\phi_i}{\hat u}$ for each $i\in \{1,\cdots, k\}$ and $\Phi = (\hat \phi_1,\cdots, \hat \phi_k)$, which is well defined in $M$ thanks to Claim \ref{HE}. Then $|\Phi|\le 1$ everywhere in $M$. Furthermore, because of (\ref{HE1}), the metric $g_{\hat u}$ is scalar flat in $M$, and thus
\begin{equation}\label{HM1}
-\Delta_{g_{\hat u}}(\hat \phi_i) = L_{g_{\hat u}}(\hat \phi_i) =\hat u^{-\frac{n+2}{n-2}}L_g(\phi_i) = -\hat u^{-\frac{n+2}{n-2}} \Delta_{g}\phi_i = 0.
\end{equation}
This shows that $\Phi$ is $g_{\hat u}$-harmonic. On the other hand, from (\ref{l11}), we have that $\sum_{i=1}^k\hat\phi_i^2 = 1$ along $\Sigma^*$, that is, $\Phi(x)\in \mathbb{S}^{k-1}(1)$ for a.e. $x\in \Sigma$. Thus $\Phi(M)\subset \mathbb{B}^k$ and $\Phi(\Sigma^*)\subset \mathbb{S}^{k-1}$. Additionally, notice that as a consequence of the conformal invariance (\ref{RCL}) and of (\ref{l9}), we obtain
\begin{equation}
B_{g_{\hat u}}(\hat \phi_i) = u^{1-N}B_g(\phi_i) = u^{1-N} \lambda_2 \phi_i u^{N-2} = \lambda_2 \hat \phi_i,
\end{equation}
on $\Sigma^*$, and therefore
\begin{equation}
\partial_{\nu_{g_{\hat u}}} \hat \phi_i =  (\lambda_2 - 2c_n h_{g_{\hat u}})\hat \phi_i
\end{equation}
holds a.e. on $\Sigma$. Hence, $\partial_{\nu_{g_{\hat u}}} \Phi$ is parallel to $\Phi$ on $\Sigma^*$, that is, a.e. on $\Sigma$. 
\end{proof}

It remains to prove that $\Phi$ is a $g_{\hat u}$-weakly free-boundary harmonic map. Recall that $g_{\hat u}$ may be degenerate on a zero measure set on $\Sigma$. Because of this, our testing space is $L^\infty(M)\cap W^{1,2}(M,g_{\hat u};\mathbb{R}^k)$: these are maps $V: M\to \mathbb{R}^k$ bounded in $M$ for which 
\begin{equation}
\int_M \left(|V|^2\hat u^{\frac{2n}{n-2}}+ |\nabla_g V|^2\hat u^2\right)\;dv_g <\infty. 
\end{equation}
Notice that the trace of a map in $W^{1,2}(M,g_{\hat u};\mathbb{R}^k)$ might not be well-defined globally. However, for $V=(V_1,\cdots, V_k)\in L^\infty(M)\cap W^{1,2}(M,g_{\hat u};\mathbb{R}^k)$, each $V_i\phi_i$ and each $\hat u V_i$ are in $W^{1,2}(M)$, and so their traces are well defined. This is all we need. 

\begin{claim}\label{Weakly-FHarmonic}
The map $\Phi$ constructed in Claim \ref{HM} is a $g_{\hat u}$-weakly free-boundary harmonic map.
\end{claim}
\begin{proof}
We show that $\Phi = (\hat \phi_1, \cdots, \hat \phi_k) = \frac{\Psi}{\hat u}\in W^{1,2}(M,g_{\hat u};\mathbb{R}^{k})$. First, notice that $\int_M |\nabla_{g_{\hat u}}\Phi|^2\;dv_{g_{\hat u}} = \int_M \hat u^2 |\nabla_{g}\Phi|^2\;dv_g$. On the other hand, $\nabla_g\hat \phi_i = \frac{\nabla_g \phi_i}{\hat u} - \frac{\phi_i\nabla_g \hat u}{\hat u^2}$, and so 
\begin{equation}
\hat u^2 |\nabla_{g}\Phi|^2  = |\nabla_g \Psi|^2 + |\Phi|^2|\nabla_g\hat u|^2 - 2\sum_{i=1}^k \hat \phi_i g(\nabla_g\phi_i , \nabla_g \hat u). 
\end{equation}
Therefore, 
\begin{equation}
\int_M |\nabla_{g_{\hat u}}\Phi|^2\;dv_{g_{\hat u}} \le 2 \int_M \left(|\nabla_g \Psi|^2 + |\nabla_g\hat u|^2\right)\;dv_{g} <\infty. 
\end{equation}
Finally, suppose we are given $V\in L^\infty(M)\cap W^{1,2}(M,g_{\hat u};\mathbb{R}^k)$ with $g_E(V,\Phi)=0$ a.e. on $\Sigma$. We want to show that 
\begin{equation}\label{Weakly-FHarmonic0}
\begin{split}
\int_M g_{\hat u}(\nabla_{g_{\hat u}}\Phi, \nabla_{g_{\hat u}}V) \;dv_{g_{\hat u}} & = \sum_{i=1}^k\int_M\Big{(}\hat u g(\nabla_g\phi_i,\nabla_gV_i) \\ &\hspace{.75in}- \phi_ig(\nabla_g\hat u, \nabla_g V_i)\Big{)}\;dv_{g}
\end{split}
\end{equation}
is zero. Since $V_i\phi_i, \hat u V_i\in W^{1,2}(M)$ for each $i = 1,\cdots, k$, we can test $(\hat u V_i)$ in $\Delta_g\phi = 0$ and $V_i\phi_i$ in $\Delta_g\hat u = 0$, to obtain
\begin{equation}\label{Weakly-FHarmonic1}
\sum_{i=1}^k\int_M\left(V_ig(\nabla_g \phi_i, \nabla_g \hat u) + \phi_i g(\nabla_g V_i, \nabla_{g} \hat u)\right)\;dv_g = 0 
\end{equation}
and 
\begin{equation}\label{Weakly-FHarmonic2}
\sum_{i=1}^k\int_M\left(V_ig(\nabla_g \phi_i, \nabla_g \hat u) + \hat u g(\nabla_g V_i, \nabla_{g} \phi_i)\right)\;dv_g = 0,
\end{equation}
where we have used that both $g_E(V,\Phi) = 0$ and $\partial_{\nu_{g_{\hat u}}}\Phi \| \Phi$ a.e. on $\Sigma$. Subtracting (\ref{Weakly-FHarmonic1}) from (\ref{Weakly-FHarmonic2}) yields (\ref{Weakly-FHarmonic0}), which finishes the proof. 
\end{proof}


\newpage
\begin{thebibliography}{9}

\bibitem{Almaraz} S. Almaraz, Convergence of scalar-flat metrics on manifolds with boundary under a Yamabe-type flow, \textit{J. Differential Equations} \textbf{259} (2015), no. 6, 2626--2694.

\bibitem{Ammann}
B. Ammann and E. Humbert, The second Yamabe invariant,
\textit{J. Funct. Anal.} \textbf{235} (2006), no. 2, 377--412.

\bibitem{Brezis}
H. Brezis, \textit{Functional Analysis, Sobolev Spaces and Partial Differential Equations},
Universitext, Springer, New York, 2011.

\bibitem{Chen} S.-Y. S. Chen, Conformal deformation to scalar flat metrics with constant mean curvature on the boundary in higher dimensions, \textit{Geom. Funct. Anal.} \textbf{19} (2009), 1029--1064.

\bibitem{ChenLaiWang} X. Chen, M. Lai, and F. Wang, Escobar--Yamabe compactifications for Poincar\'e--Einstein manifolds and rigidity theorems, \textit{Adv. Math.} \textbf{343} (2019), 16--35.

\bibitem{ClappPellacciPistoia}
M. Clapp, B. Pellacci, and A. Pistoia, Sign-changing solutions to the Yamabe problem on manifolds with boundary, arXiv:2511.10553v2 [math.DG], 2026.

\bibitem{Cox}
G. Cox, D. Jakobson, M. Karpukhin, and Y. Sire,
Conformal invariants from nodal sets. II. Manifolds with boundary,
\textit{J. Spectr. Theory} \textbf{11} (2021), no. 2, 387--409.

\bibitem{ElSayed}
S. El Sayed, Second eigenvalue of the Yamabe operator and applications,
\textit{Calc. Var. Partial Differential Equations} \textbf{50} (2014),
no. 3--4, 665--692.

\bibitem{Soufi}
A. El Soufi and S. Ilias,
Riemannian manifolds admitting isometric immersions by their first eigenfunctions,
\textit{Pacific J. Math.} \textbf{195} (2000), no. 1, 91--99.

\bibitem{Soufi2}
A. El Soufi and S. Ilias,
Laplacian eigenvalue functionals and metric deformations on compact manifolds,
\textit{J. Geom. Phys.} \textbf{58} (2008), no. 1, 89--104.

\bibitem{Escobar}
J. F. Escobar,
Uniqueness theorems on conformal deformation of metrics, Sobolev inequalities,
and an eigenvalue estimate,
\textit{Comm. Pure Appl. Math.} \textbf{43} (1990), no. 7, 857--883.

\bibitem{Escobar2}
J. F. Escobar,
Conformal deformation of a Riemannian metric to a scalar flat metric with
constant mean curvature on the boundary,
\textit{Ann. of Math. (2)} \textbf{136} (1992), no. 1, 1--50.

\bibitem{Escobar3}
J. F. Escobar,
Addendum: conformal deformation of a Riemannian metric to a scalar flat metric
with constant mean curvature on the boundary,
\textit{Ann. of Math. (2)} \textbf{139} (1994), no. 3, 749--750.

\bibitem{Fraser}
A. Fraser and R. Schoen,
Minimal surfaces and eigenvalue problems,
in \textit{Geometric Analysis}, Contemp. Math., vol. 599,
Amer. Math. Soc., Providence, RI, 2013, pp. 105--121.

\bibitem{Fraser2}
A. Fraser and R. Schoen,
Sharp eigenvalue bounds and minimal surfaces in the ball,
\textit{Invent. Math.} \textbf{203} (2016), no. 3, 823--890.

\bibitem{Gilbarg}
D. Gilbarg and N. S. Trudinger,
\textit{Elliptic Partial Differential Equations of Second Order},
2nd ed., Classics in Mathematics, Springer-Verlag, Berlin, 2001.

\bibitem{GonzalezSaez}
M. del Mar Gonz\'alez and M. S\'aez,
Eigenvalue bounds for the Paneitz operator and its associated third-order boundary operator on locally conformally flat manifolds,
\textit{Ann. Sc. Norm. Super. Pisa Cl. Sci. (5)}
\textbf{24} (2023), no. 4, 2047--2081.

\bibitem{Gursky}
M. J. Gursky and S. P\'erez-Ayala,
Variational properties of the second eigenvalue of the conformal Laplacian,
\textit{J. Funct. Anal.} \textbf{282} (2022), no. 8, Paper No. 109371.

\bibitem{Hamanaka}
S. Hamanaka and P. T. Ho,
Notes on the uniqueness of Type II Yamabe metrics,
\textit{Nonlinear Differ. Equ. Appl.} \textbf{32} (2025), Article No. 81.

\bibitem{Hebey}
E. Hebey,
\textit{Nonlinear Analysis on Manifolds: Sobolev Spaces and Inequalities},
Courant Lecture Notes in Mathematics, vol. 5,
New York University, Courant Institute of Mathematical Sciences, New York;
American Mathematical Society, Providence, RI, 2000.

\bibitem{Hsiao}
G. C. Hsiao and W. L. Wendland,
\textit{Boundary Integral Equations},
Applied Mathematical Sciences, vol. 164,
Springer-Verlag, Berlin, 2008.

\bibitem{HumbertPetridesPremoselli}
E. Humbert, R. Petrides, and B. Premoselli, Extremising eigenvalues of the GJMS operators in a fixed conformal class, arXiv:2505.08280 [math.DG], 2025.

\bibitem{Jost}
J. Jost, L. Liu, and M. Zhu,
The qualitative behavior at the free boundary for approximate harmonic maps
from surfaces,
\textit{Math. Ann.} \textbf{374} (2019), no. 1--2, 133--177.


\bibitem{KarpukhinNadirashviliPenskoiPolterovich}
M. Karpukhin, N. Nadirashvili, A. V. Penskoi, and I. Polterovich,
Conformally maximal metrics for Laplace eigenvalues on surfaces,
\textit{Surv. Differ. Geom.} \textbf{24} (2019), 205--256.

\bibitem{KarpukhinStern}
M. Karpukhin and D. L. Stern, Min--max harmonic maps and a new characterization of conformal eigenvalues, \textit{J. Eur. Math. Soc. (JEMS)} \textbf{26} (2024), no. 11, 4071--4129.

\bibitem{Kokarev}
G. Kokarev,
Variational aspects of Laplace eigenvalues on Riemannian surfaces,
\textit{Adv. Math.} \textbf{258} (2014), 191--239.

\bibitem{Lassas}
M. Lassas, T. Liimatainen, and M. Salo,
The Calder\'on problem for the conformal Laplacian,
\textit{Comm. Anal. Geom.} \textbf{30} (2022), no. 5, 1121--1184.

\bibitem{LaurainPetrides}
P. Laurain and R. Petrides,
Regularity and quantification for harmonic maps with free boundary,
\textit{Adv. Calc. Var.} \textbf{10} (2017), no. 1, 69--82.

\bibitem{Li}
Z. Li, Unique continuation for Robin problems with non-smooth potentials, \textit{J. Funct. Anal.} \textbf{288} (2025), no. 6,
Article No. 110811.

\bibitem{WinkertBoundedness}
P. Winkert,
On the boundedness of solutions to elliptic variational inequalities,
\textit{Set-Valued Var. Anal.} \textbf{22} (2014), no. 4, 763--781.

\bibitem{Marino}
G. Marino and P. Winkert,
Moser iteration applied to elliptic equations with critical growth on the boundary,
\textit{Nonlinear Anal.} \textbf{180} (2019), 154--169.

\bibitem{MarquesCAG} F. C. Marques, Conformal deformations to scalar-flat metrics with constant mean curvature on the boundary, \textit{Comm. Anal. Geom.} \textbf{15} (2007), no. 2, 381--405.

\bibitem{Marques} F. C. Marques, Existence results for the Yamabe problem on manifolds with boundary,
\textit{Indiana Univ. Math. J.} \textbf{54} (2005), no. 6, 1599--1620.

\bibitem{Nadirashvili}
N. Nadirashvili,
Berger's isoperimetric problem and minimal immersions of surfaces,
\textit{Geom. Funct. Anal.} \textbf{6} (1996), no. 5, 877--897.


\bibitem{Petrides}
R. Petrides,
Existence and regularity of maximal metrics for the first Laplace eigenvalue
on surfaces,
\textit{Geom. Funct. Anal.} \textbf{24} (2014), no. 4, 1336--1376.

\bibitem{Petrides2}
R. Petrides,
Maximizing Steklov eigenvalues on surfaces,
\textit{J. Differential Geom.} \textbf{113} (2019), no. 1, 95--188.

\bibitem{PerezAyalaSireXu}
S. P\'erez-Ayala, Extremal eigenvalues of the conformal Laplacian under Sire--Xu normalization, \textit{Nonlinear Anal.} \textbf{208} (2021), Article No. 112308.

\bibitem{PerezAyalaPaneitz}
S. P\'erez-Ayala, Extremal metrics for the Paneitz operator on closed four-manifolds, \textit{J. Geom. Phys.} \textbf{182} (2022), Article No. 104666.

\bibitem{TaylorPDO} M. E. Taylor, \textit{Pseudodifferential Operators and Nonlinear PDE}, Progress in Mathematics, vol. 100, Birkh\"auser, Boston, 1991.

\bibitem{VinokurovSymmetries} D. Vinokurov, Conformal optimization of eigenvalues on surfaces with symmetries, \textit{J. Lond. Math. Soc. (2)} \textbf{112} (2025), no. 6, e70386.

\bibitem{VinokurovHigher} D. Vinokurov, Maximizing higher eigenvalues in dimensions three and above, arXiv:2506.09328 [math.SP], 2025.




\end{thebibliography}
\end{document}